\documentclass[10pt]{amsart}
\usepackage{graphicx}
\graphicspath{{figs/}}
\usepackage{amsthm,amsfonts,amssymb,amsmath}
\usepackage[centerlast]{subfigure}
\usepackage[rotate,arrow,matrix]{xy}
\usepackage{tikz}

\DeclareMathOperator{\codim}{codim}

\DeclareMathOperator{\ord}{ord}

\DeclareMathOperator{\surf}{Surf}

\DeclareMathOperator{\bl}{Bl}

\newtheorem{thm}{Theorem}[section]
\newtheorem{lem}[thm]{Lemma}
\newtheorem{coro}[thm]{Corollary}
\newtheorem{prop}[thm]{Proposition}

\theoremstyle{definition}
\newtheorem{df}[thm]{Definition}
\newtheorem{ex}[thm]{Example}

\theoremstyle{remark}
\newtheorem{rem}[thm]{Remark}

\def\ol#1{\overline{#1}}
\def\ul#1{\underline{#1}}
\def\v#1{\overrightarrow{#1}}

\def\N{\mathbb N}
\def\Z{\mathbb Z}

\def\R{\mathbb R}
\def\C{\mathbb C}
\def\S{\mathbb S}
\def\TO{\mathbb T}
\def\FS{\mathfrak S}

\def\M{\mathcal M}

\def\Md{\M^{dec}}
\def\Mc{\M^{comb}}
\def\MgP{\M_{g,P}}
\def\MU#1#2{\ul{\M}_{#1,#2}}
\def\MO#1#2{\ol{\M}_{#1,#2}}
\def\MDgP{\M_{g,P}^{dec}}
\def\MCgP{\M_{g,P}^{comb}}
\def\MUC#1#2{\ul{\M}_{#1,#2}^{comb}}
\def\MOC#1#2{\ol{\M}_{#1,#2}^{comb}}
\def\MUD#1#2{\ul{\M}_{#1,#2}^{dec}}
\def\MOD#1#2{\ol{\M}_{#1,#2}^{dec}}

\begin{document}

\title[Compactifications of Moduli Spaces]{Compactifications of Moduli Spaces and Cellular Decompositions}
\address {Departmento de Econom\'{i}a\\Universidad del Pac\'{i}fico\\
Av. Salaverry 2020, Jes\'{u}s Mar\'{i}a, Lima 11 - Per\'{u}}
\author[J. Z\'{u}\~{n}iga]{Javier Z\'{u}\~{n}iga}
\email{zuniga\underline{ }jj@up.edu.pe}
\date{\today}

\begin{abstract}
This paper studies compactifications of moduli spaces involving closed Riemann surfaces. The first main result identifies the homeomorphism types of these compactifications. The second main result introduces orbicell decompositions on these spaces using semistable ribbon graphs extending the earlier work of Looijenga.
\end{abstract}

\maketitle

\tableofcontents

\section{Introduction}

By a Riemann surface or simply a curve we mean a compact connected complex manifold of complex dimension one. Denote by $\M_{g,n}$ the moduli space of Riemann surfaces with genus $g$ and $n>0$ labeled points. The Deligne-Mumford compactification is denoted by $\MO{g}{n}$. This is a space parameterizing stable Riemann surfaces. Here the word ``stable" refers to the finiteness of the group of conformal automorphism of the surface. Geometrically it means that we only allow double point (also called node) singularities and that each irreducible component of the surface has negative Euler characteristic (taking the labeled points and nodes into account). We can further perform a real oriented blowup along the locus of degenerate surfaces to obtain the space $\MU{g}{n}$. Intuitively, this space is similar to the Deligne-Mumford space but it also remembers the angle at each double point at which the surface degenerated.

The decorated moduli space is denoted by $\Md_{g,n}=\M_{g,n} \times \Delta^{n-1}$ where $\Delta^{n-1}$ is the $(n-1)$ dimensional standard simplex. The decorations can be thought of as hyperbolic lengths of certain horocycles or as quadratic residues of Jenkins-Strebel differentials on a Riemann surface. By choosing an appropriate notion of decoration on a stable Riemann Surface it is possible to construct compactifications $\MOD{g}{n}$ and $\MUD{g}{n}$. The first main result of this paper identifies the homeomorphism type of these compactifications. Let $P$ be a finite set of labels.

\medskip
\noindent {\bf Corollary ~\ref{maincoro1}.} \[ \MUD{g}{P} \cong \MU{g}{P} \times \Delta_P  \] \smallskip \emph{and therefore $\MUD{g}{P}$ is Hausdorff and compact.}

\medskip
\noindent {\bf Theorem ~\ref{mainthm1}.} \emph{There is a map $\MOD{g}{P} \to \MO{g}{P} \times \Delta_P$ which is a homeomorphism in the interior and has conical singularities along the boundary of $\MOD{g}{P}$ and is thus a homotopy equivalence.}
\medskip

It is a known result of Harer, Mumford, Thurston \cite{har86}, Penner \cite{pen87}, Bowditch, Epstein \cite{boweps}, that the decorated moduli space is homeomorphic to the moduli space of metric ribbon graphs denoted by $\Mc_{g,n}$. This later space comes with a natural orbi-cellular structure given by ribbon graphs. In \cite{kon:itma} Kontsevich introduces a way to compactify this space in order to prove Witten's conjecture. Later on Looijenga formalized and extended these ideas in \cite{loo} in connection with the arc complex. The second part of this paper describes a cellular compactification of the ribbon graph space, extending the work of Looijenga. The main results are the following.

\medskip
\noindent {\bf Theorem ~\ref{hmuc}.} \emph{The map $\ul{\Psi}: \MUC{g}{P} \to \MUD{g}{P}$ is a homeomorphism.}

\medskip
\noindent {\bf Theorem ~\ref{mainthm3}.} \emph{The map $\ol{\Psi}: \MOC{g}{P} \to \MOD{g}{P}$ is a homeomorphism.}
\medskip

This new compactification covers Looijenga's and Kontsevich's compactifications and is finer, meaning that it encodes more information. It also seems more relevant to quantum field theory purposes. In particular, it should be possible to describe a BV structure on the cellular chains of our compactification and construct a solution to the quantum master equation in a future work. This solution is purely combinatorial and so it avoids the use of string vertices or geometric chains.

I would like to thank Sasha Voronov for his generosity and guidance, Eduard Looijenga for his patience answering my questions, and Kevin Costello for sharing his own ideas about this work with me. I am also grateful to Jim Stasheff for reviewing an early draft of this paper. Finally I would like to acknowledge the enormous contribution made by the referee to the quality and clarity of the present exposition. 

\section{Real Oriented Blowups}

\begin{figure} 
\includegraphics[width=8cm]{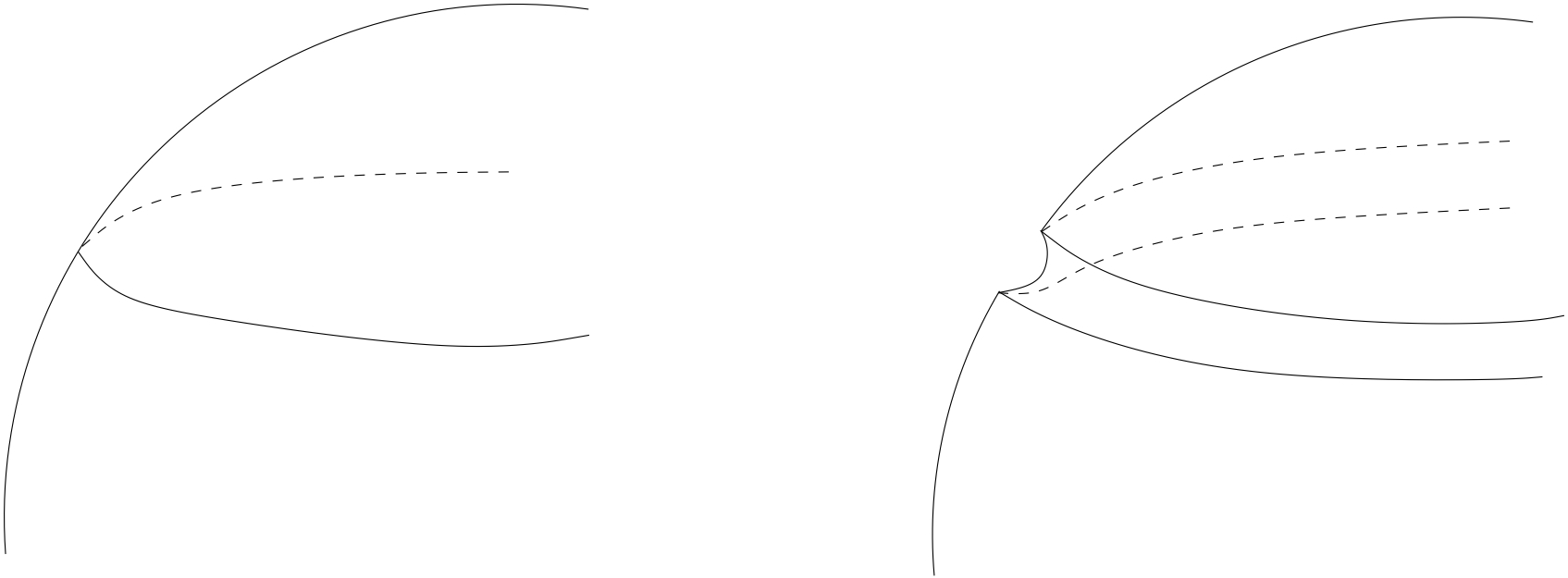}
\caption{On the left: A three dimensional manifold with a one dimensional submanifold of the boundary. On the right: its real oriented blowup.} \label{bbl}
\end{figure}

We will use a blowup construction in the PL category. Given a manifold $M$ and a closed submanifold $N$ the real (or directional) oriented blowup $\bl_N(M)$ can be defined by gluing $M -N$ to the ($\codim N$-1)-dimensional spherical bundle of rays of the normal bundle of $N$ in $M$. This is homeomorphic to the result of carving an open tubular neighborhood of $N$ out of $M$. There is a natural projection map $\bl_N(M) \to M$. The construction can be generalized to the PL category of manifolds with boundary and the submanifold $N$ can be replaced by a union of submanifolds with some transversality condition.

\begin{lem} 
Blowing up a submanifold of the boundary of a manifold does not change the homeomorphism type of the original manifold. \label{busb}
\end{lem}

\begin{proof}
The normal bundle of a submanifold in the boundary of $M$ is a closed half space bundle. Therefore the bundle of rays is a half sphere bundle. This process enlarges the boundary of $M$ without changing its homeomorphism type as in Figure~\ref{bbl}. A homeomorphism can be realized by using a tubular neighborhood of the submanifold.
\end{proof}

Given a union of PL-submanifolds intersecting multi-transversely, it will be sometimes necessary to blow up such union with the aid of a filtration indexed by dimension. In this case we will blow up from the lowest dimensional to the highest dimensional elements of the filtration. We will denote by $\bl_F(M)$ the sequential blowup of $M$ along the filtration $F=\{ P_i \}$ indexed by dimension. An example can be seen in Figure~\ref{blsb}.

In what follows the symbol ``$\cong$" means homeomorphic and ``$\simeq$" means homotopic.

\begin{figure} 
\includegraphics[width=12.5cm]{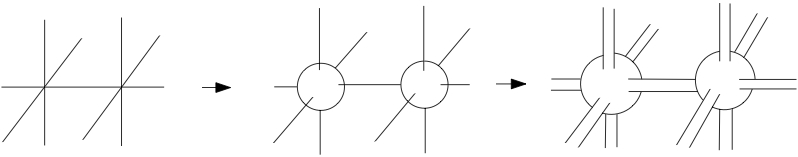}
\caption{$\bl_{\{ \{(0,0,0)\}, \{(1,0,0)\}, \{(0,0,z)\}, \{(0,y,0)\}, \{(x,0,0)\}, \{(1,0,z)\}, \{(1,y,0)\} \}}(\R^3)$} \label{blsb}
\end{figure}

\begin{lem}
Given two manifolds $X$,$Y$ and a submanifold $Z \subset X$ we have \[ \bl_{Z\times Y}(X \times Y) \cong \bl_Z(X) \times Y \] \label{blowupxprod} 
\end{lem} 

\begin{proof} Let $T(X)$ denote the tangent bundle of $X$ and $\nu_X(Z)$ the normal bundle of $Z$ in $X$. Since $T(Z \times Y) \cong T(Z) \times T(Y)$ the vectors normal to $Z \times Y$ in $T(X \times Y) \cong T(X) \times T(Y)$ do not include any vector in the second factor. Thus $\nu_{X \times Y}(Z\times Y) \cong \nu_X(Z) \times Y$ and the result follows.
\end{proof}

Let $\S^{2n-1}=\{ \v{z} \in \C^n | \, \sum |z_i|^2 =1\}$ and $\TO^n = (\S^1)^n$. We also denote by $B^n$ the {\bf $n$-dimensional open unit ball} $\{\v{x} \in \R^n | \, |\v{x}| < 1 \}$ and by $\ol{B^n}$ its closure in $\R^n$.

\begin{lem}
Let $\TO^n$ act on $\S^{2n-1}$ by $(\theta_1,...\theta_n) \cdot (z_1,...,z_n) = (\theta_1 z_1,..., \theta_n z_n)$. Then \[\S^{2n-1} / \TO^n \cong \Delta^{n-1}\] \label{sphereovertorus}
\end{lem}

\begin{proof}
Notice that $\bl_{\{\v{0}\}}(\C^n) \cong \C^n - B^{2n}$ by extending to the boundary the homeomorphism taking $\v{x}$ to $\v{x} + \frac{\v{x}}{|\v{x}|}$. Thus the boundary generated is isomorphic to $\S^{2n-1}$. Consider the map $\pi : \S^{2n-1} \to \Delta^{n-1}$ defined by $(z_1,...,z_n) \mapsto (|z_1|^2,...,|z_n|^2)$. The preimage of a point in the simplex corresponds with the orbit of the torus action and hence the map descends to the desired homeomorphism after taking the quotient.
\end{proof}

In order to clarify the homeomorphism type of certain quotient space we introduce the notion of conical singularity. If $X$ is a topological space let $CX$ be its cone and $v$ the vertex. When $X$ is a topological manifold the resulting cone is locally Euclidean in $CX-\{ v \}$. When $CX$ is not locally Euclidean at $v$ we call this vertex a {\bf conical singularity}. If $X$ and $Y$ are topological manifolds and $v$ is a conical singularity of $CX$ then any point in $\{ v \} \times Y \subset CX \times Y$ is also called a conical singularity.

\begin{lem}
As in Lemma~\ref{sphereovertorus}, the torus $\TO^n$ acts on the boundary of $\bl_{\{\v{0}\}}(\C^n)$. The quotient $\bl_{\{\v{0}\}}(\C^n) / \TO^n$ is the disjoint union of $\C^n - \{ \v{0} \}$ and $\Delta^{n-1}$ and has conical singularities along the second subspace. This space is contractible and hence homotopic to $\C^n$.
\end{lem}

\begin{proof}
We can view the quotient $\bl_{\{\v{0}\}}(\C^n) / \TO^n$ as the union of $\C^n - \{ \v{0} \}$ with $\Delta^{n-1}$ due to Lemma~\ref{sphereovertorus}. Therefore this quotient can be understood as an enlargement of the origin into the simplex. However this enlargement is not homeomorphic to $\C^n$. To see this take a point $p$ in the interior of the simplex. A neighborhood of this point in the simplex is homeomorphic to $B^{n-1}$. The preimage under $\pi: \S^{2n-1} \to \Delta^{n-1}$ of each point in this neighborhood under the torus action is $\TO^n$. The normal bundle has rank one (it is the bundle of rays in $\C^n$ orthogonal to the sphere $\S^{2n-1}$). Thus a neighborhood of $p$ is homeomorphic to $(B^{n-1} \times [0,1) \times \TO^n) / \TO^n$ where the semi-open interval $[0,1)$ corresponds with the normal bundle and the torus acts only at level zero (which corresponds to the simplex). At $p$ this quotient gives the cross product of the torus with the interval where one end is collapsed to a point. But this is exactly $(C \TO^n)^\circ$, \emph{i.e.} the interior of the cone over the torus. Thus a neighborhood of $p$ in $\bl_{\{\v{0}\}}(\C^n) / \TO^n$ is homeomorphic to $B^{n-1} \times (C \TO^n)^\circ$ and $p$ is a conical singularity. We will refer to these open sets as {\bf toric neighborhoods}. For a point on the boundary of the simplex a similar (but more careful) analysis yields the same toric neighborhoods. Away from the simplex the quotient of the blowup is still homeomorphic to $\C^n$.

Contracting radially gives a $\TO^n$-equivariant deformation retraction of  $\bl_{\{\v{0}\}}(\C^n)$ onto its boundary:  $\S^{2n-1}$. Therefore the quotient deformation retracts onto $\Delta^{n-1}$ which is contractible. In particular, this space is homotopic to $\C^n$.
\end{proof}

\begin{coro}
Denote by $D_m$ the intersection of the $m$ complex hyperplanes in $\C^n$ given by $z_i=0$ where $1\le i \le m$. The torus $\TO^m$ acts on the boundary of $\bl_{D_m}(\C^n)$ and in fact the quotient is homeomorphic to $\left( \bl_{\{ \v{0} \}} \C^m \right) / \TO^m \times \C^{n-m}$. In particular, it has conical singularities at the points of the simplex in the first factor and it is contractible.
\label{torichood}
\end{coro}

\begin{proof}
$D_m$ can be described as $\{(0,...,0,z_{m+1},...,z_n) \} \cong \{\v{0}\} \times \C^{n-m}$ and $\C^n \cong \C^m \times \C^{n-m}$ . By Lemma~\ref{blowupxprod} \[ \bl_{\{ \v{0} \} \times \C^{n-m}} (\C^{m} \times \C^{n-m}) \cong \left( \bl_{\{ \v{0} \}} \C^m \right) \times \C^{n-m} \] and therefore we can apply the previous lemma to the first component since $\TO^n$ only acts on the first factor.
\end{proof}

Consider the $n-1$ dimensional standard simplex $\Delta^{n-1}$ with vertices labeled by the set $[n]=\{1,...,n\}$. Since every face can be identified with a subset of $[n]$ given a subset $P \subset 2^{[n]}$ we denote by $\bl_P(\Delta^{n-1})$ the blowup of the simplex along the filtration indexed by dimension obtained from the faces induced by $P$. Notice that blowing up along sets with $n-1$ or $n$ elements does not change the geometry in a meaningful way since the result is linearly isomorphic to the simplex. Therefore the only contributions come from blowing up faces of codimension at least two.

\begin{figure} 
\includegraphics[width=3.9cm]{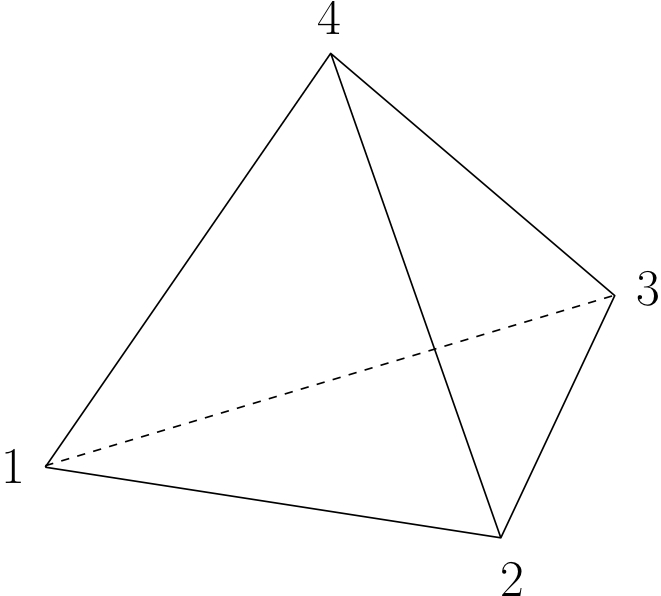} \quad \qquad
\includegraphics[width=3.7cm]{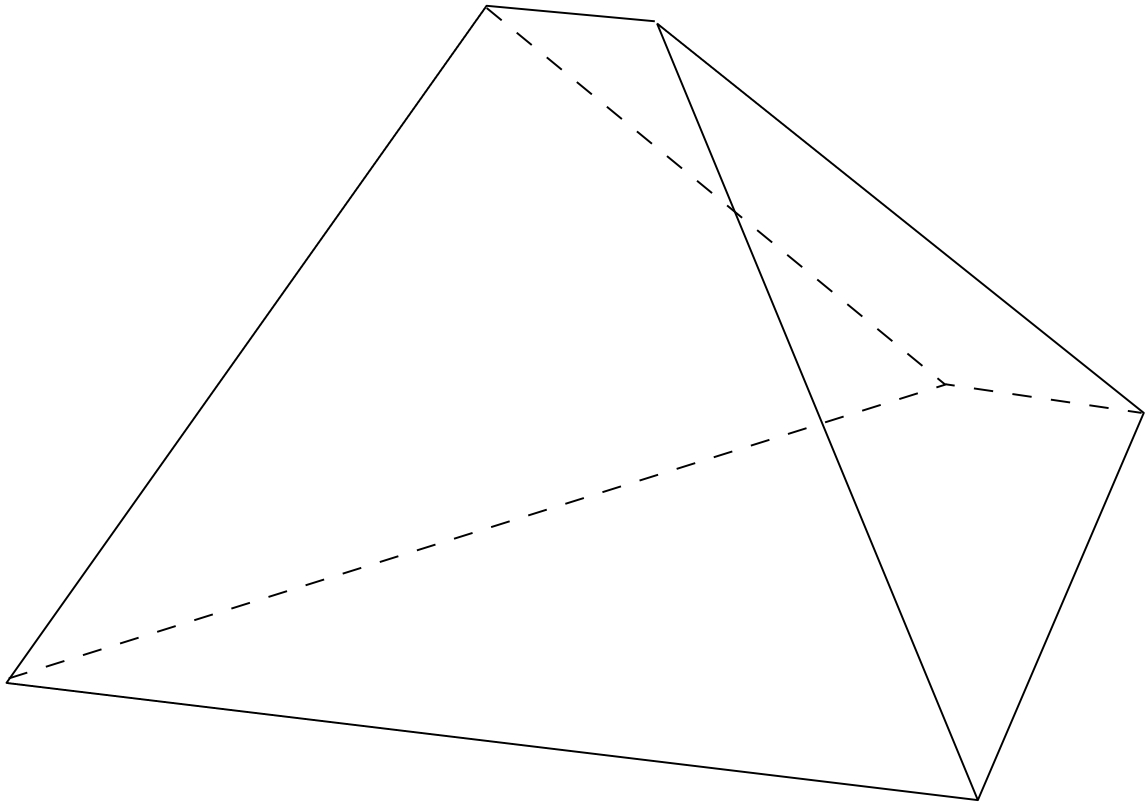} \par \bigskip \bigskip
\includegraphics[width=3.8cm]{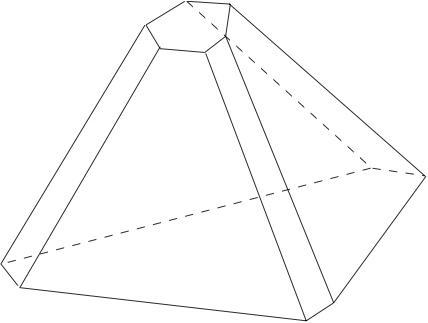} \quad \qquad
\includegraphics[width=3.7cm]{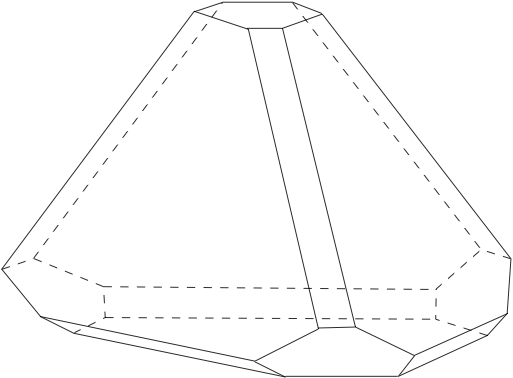}
\caption{The blowup of $\Delta^3$ along the sets $\varnothing$, \{\{3,4\}\}, \{\{4\}, \{1,4\}, \{2,4\}, \{3,4\}\}, and \{\{1\}, \{2\}, \{3\}, \{4\}, \{1,2\}, \{1,3\}, \{1,4\}, \{2,3\}, \{2,4\}, \{3,4\}\}.} \label{bu}
\end{figure}

\begin{rem}
If $P=2^{[n]}$ then $\bl_P (\Delta^{n-1})$ produces the $n-1$ dimensional cyclohedron. This can be made explicit using for example Proposition 4.3.1 in \cite{deva}. The associahedra can also be obtained as a blowup of standard simplices. For instance the three dimensional associahedron ${\mathcal K}^3$ corresponds with the blowup of $\Delta^3$ along $P=\{ \{2\}, \{3\}, \{1,2\}, \{2,3\}, \{3,4\} \}$.
\end{rem}

In what follows it will be necessary to consider the real oriented blowup in the category of orbifolds. Locally, an orbifold looks like $\R^n / G$ where $G$ is a finite group acting linearly. This allow us to define the real oriented blowup along a subspace of the quotient by first blowing up the orbit in $\R^n$ and then taking the quotient by the induced action. The compatibility conditions of the orbifold define the blowup globally.

\section{Decorated Moduli Spaces}

\subsection{Compactifications} 

The moduli space $\MgP$ parametrizes conformal classes of Riemann surfaces of genus $g$ with a fixed finite subset $P$ of labeled points. We will also denote this moduli space as $\M_{g,n}$ where $n=|P|$. The {\bf topological type} of a surface is defined as the pair $(g,n)$. The symmetric group $\FS_P$ acts by permuting the labels. A {\bf decoration} of a labeled point is a non-negative real number associated to that point. We require that each labelled point has a decoration and that the total sum of decorations is one. This gives $\MgP \times \Delta_P$ where $\Delta_P$ is the $|P|-1$ dimensional standard simplex spanned by $P$. Denote this space by $\MDgP$ and call these surfaces {\bf $P$-labeled Riemann surfaces of genus $g$ decorated by real numbers}. Its dimension is $6g+3n-7$. In \cite{mupe} there is a description of an orbicell decomposition for $\MDgP$ and $\MDgP / \FS_P$ in terms of ribbon graphs. The aim of this paper is to construct orbicell decompositions for compactified versions of these spaces using ribbon graphs and relate their homeomorphism types to those of $\MU{g}{n}$ and $\MO{g}{n}$.

For $\MgP$ it is possible to take the Deligne-Mumford compactification $\MO{g}{P}$ which parametrizes isomorphism classes of $P$-labeled stable Riemann surfaces. We can further perform a real oriented blowup along the locus of stable curves with singularities to obtain the moduli space $\MU{g}{P}$ as in \cite{ksv} which is called the {\bf moduli space of $P$-labeled stable Riemann surfaces decorated by real tangent directions}. To better understand this space consider the normal bundles to the locus of stable curves with singularities. Locally, when we have only one singularity, the normal bundle is canonically isomorphic to the tensor product of two tangent spaces of the surface, one for each side of the singularity. Points in the boundary of the real oriented blowup then correspond to real rays in the tensor product. This information encodes an angle at each double point of the surface and all possible angles describe a circle. The natural projection $\MU{g}{P} \to \MO{g}{P}$ has as preimages finite quotients of real tori (a product of circles) on the locus of singular curves. The dimension of the torus is equal to the number of singularities and the group action is induced by conformal automorphisms.

We now introduce a way to compactify $\MDgP$ motivated by \cite{loo}. A $P$-labeled nodal Riemann surface $C$ is {\bf semistable} when its irreducible components minus labels and nodes have non-positive Euler characteristic. Denote by $\hat{C} = \sqcup_{i \in I} C_i$ its normalization where the $C_i$'s are connected and irreducible. The preimages of singularities under the attaching map are called {\bf nodes}. Let $N$ be the set of nodes and $\iota: N \to N$ the induced involution. Two elements of $N$ are {\bf associated} if they belong to the same orbit of $\iota$. Two components of $\hat{C}$ are associated if one of them has a node associated to a node in the other component. The only smooth $P$-labeled semistable surface that is not stable ($\chi=0$) is the Riemann sphere with two labeled points which will only arise as nodes. We call this surface a {\bf semistable sphere}. Its moduli space is just a point. The only other semistable surface with zero Euler characteristic is the compact torus but since this surface has no labeled points we will not consider this case. Now we further restrict these surfaces as to comply with the following conditions. 

\begin{enumerate}
\item A component cannot be associated to itself.
\item Two semistable spheres cannot be associated.
\item The two points in a semistable sphere are always nodes.
\item A stable component with no labeled points must be associated with at least one other stable component.
\end{enumerate}

\begin{df}
A {\bf perimeter function} for $C$ is a function $\lambda: P \cup N \to [0,1]$ with the following two properties.
\begin{enumerate}
\item If $p$ and $q$ are the nodes of a semistable sphere then $\lambda(p)=\lambda(q)$. \item Every connected component of $\hat{C}$ has at least one point $p \in P \cup N$ with $\lambda(p)>0$.
\end{enumerate}
\end{df}

\begin{df}
An {\bf order} for $C$ is a function $\ord: \pi_0 \hat{C} \to \N$ where $\N = \{ 0,1,2,...\}$ with the condition that if $\ord([C_i])=k>0$ then there exist $j$ such that $\ord([C_j])=k-1$.
\end{df}

A component of order $k$ will be called a {\bf $k$-component}. We will also denote by $\hat{P}_k$ and $\hat{N}_k$ the subsets of $P$ and $N$ lying on $k$-components. 

\begin{df}
We say that the pair $(\lambda, \ord)$ is {\bf compatible} if they satisfy the following property: Let $p \in C_i \cap N$ and $q=\iota(p) \in C_j \cap N$ then $\lambda(p)>0$ if and only if $\lambda(q)=0$. Moreover, in this case we require that $\ord([C_j])<\ord([C_i])$.
\end{df}

\begin{df}
A compatible pair $(\lambda,\ord)$ is {\bf unital} if for each fixed $k$ \[ \sum_{p \in \hat{P}_k \cup \hat{N}_k} \lambda (p) = 1 \] A unital pair will be called a {\bf decoration} of $C$. This definition agrees with the definition of a decoration on smooth Riemann surfaces where there is only one component of order zero.
\end{df}

\begin{figure}
\includegraphics[width=12.5cm]{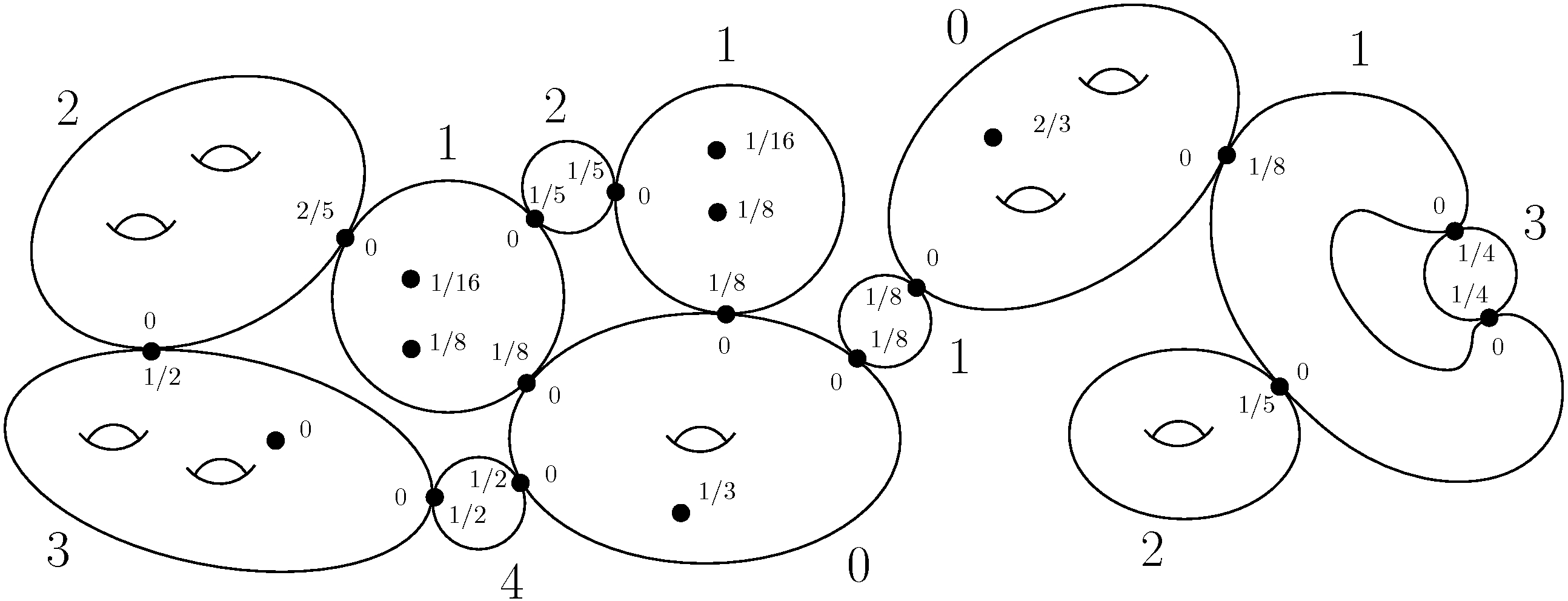}
\caption[Semistable surface with unital pair.]{Semistable surface with unital pair: the numbers inside the surface correspond to decorations by real numbers and the numbers outside the surface correspond to the orders of the irreducible components.} \label{surford}
\end{figure}

\begin{lem} Given a decoration $(\lambda,\ord)$ of $C$ we have:
\begin{enumerate}
\item If $p\in C_i \cap N$ with $\ord([C_i])=0$ then $\lambda(p)=0$.
\item There is a constant $m \in \N$ such that $\ord([C_i])\le m$ for all $i$ and given $k$ such that $0 \le k \le m$ there exist $i$ with $\ord([C_i])=k$.
\item If $p,q \in C_i \cap N$ where $C_i$ is a semistable sphere then $\lambda(p)=\lambda(q)>0$.
\item If $C_i$ is a semistable sphere then any component associated to it has a lower order.
\item A component cannot be associated to another component of the same order.
\end{enumerate}
\end{lem}

\begin{proof}
To show (1) suppose $p \in C_i \cap N$ with $\lambda(p)>0$. Then from the definition of order we get a component $C_j$ with $\ord([C_j])<\ord([C_i])=0$. But this is a contradiction since $\ord([C_i]) \ge 0$. For (2) first notice that $\pi_0 (\hat{C})$ is finite because the surface is compact. Then there exists a maximal order $m$ and the result follows from the definition of order. For (3) assume that $\lambda(p)=0$. Then from the definition of perimeter function we get $\lambda(q)=0$ which is in contradiction with the second condition for a perimeter function. Now (4) follows from (3) by applying the definition of compatible pair. Finally (5) follows solely from the definition of compatible pair.
\end{proof}

\begin{ex}
Figure~\ref{surford} illustrates a unital pair.
\end{ex}

 An {\bf isomorphism} in this context is a stable surface isomorphism preserving the labels pointwise and the decorations.

\begin{df} The set of isomorphism classes of $P$-labeled semistable Riemann surfaces together with a decoration will also be called the {\bf moduli space of decorated semistable surfaces} and will be denoted by $\MOD{g}{P}$.
\end{df}

\begin{df} In an analogous way we introduce the moduli space $\MUD{g}{P}$ by adding decorations by tangent directions at each node of a surface. Isomorphisms are required to preserve this extra data.
\end{df}

\begin{rem}
The local effect of allowing semistable components in the moduli space is minimal. For $\MOD{g}{P}$ it is only adding the combinatorics of the decoration. We will see later that geometrically it accounts for remembering how fast a geodesic vanished. For $\MUD{g}{P}$, given two associated irreducible components, the decorations by tangent directions enlarge the real dimension of that locus by one in the moduli space. Inserting a strictly semistable sphere in between these associated components also adds only one dimension to that locus in the moduli space. This is because the group of automorphisms rotates the sphere so that the real rays corresponding to the nodes on the semistable sphere are irrelevant. This is illustrated on Figure~\ref{sing}.

\begin{figure} 
\includegraphics[width=7cm]{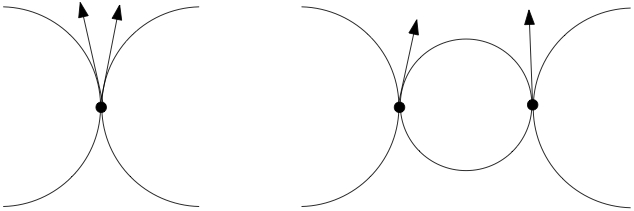}
\caption[Decorations by tangent directions.]{Decorations by tangent directions at a singularity between stable components and semistable spheres.}\label{sing}
\end{figure}
\end{rem}

\subsection{Topology} \label{compactifications} 

We wish to define a bijection $\varphi: \bl_{\bar{F}} (\MU{g}{P} \times \Delta_P) \to \MUD{g}{P}$ where $ \bl_{\bar{F}} (\MU{g}{P} \times \Delta_P)$ is certain blowup construction whose topology we understand. The topology of $\MUD{g}{P}$ is then induced via this map.

If $([C],\lambda) \in \MU{g}{P} \times \Delta_P$ then there is at least one irreducible component of $C$ that has a non-zero decoration. Call these kind of irreducible components {\bf non-zero components}, the rest will be called {\bf zero components}. Consider the different loci of singular surfaces with the following property: every node in a zero component is associated to a node on a non-zero components (notice that this rules out self-intersections on zero components since that component would be associated with itself, which is a zero component by definition). The union of all these loci defines a filtration by dimension we call $F$. The induced filtration by dimension of the closure of the previous loci is denoted by $\bar{F}$. Thus $\bar{F}$ can be viewed as the union of strata intersecting multi-transversely.

\begin{rem}
The highest dimensional strata of $\bar{F}$ correspond with the locus of surfaces with only one singularity and no irreducible component with all of their decorations equal to zero. Since the dimension of these strata is equal to the dimension of the boundary it has the same homeomorphism type and thus blowing up along these strata will not produce any new points on the boundary.
\end{rem}

\begin{thm}
There is a bijection \[ \varphi: \bl_{\bar{F}} (\MU{g}{P} \times \Delta_P) \to \MUD{g}{P}\] that induces a topology on $\MUD{g}{P}$.
\end{thm}

\begin{proof}
The map $\varphi$ is defined as the identity on $\MgP \times \Delta_P$. If $x$ belongs to the locus of singular surfaces of $\bl_{\bar{F}} (\MU{g}{P} \times \Delta_P)$ we need to define $\varphi(x)= [(C,\lambda,\ord)]$. Since $x$ is in the boundary it means that it was a point resulting from blowing up along $(\MU{g}{P} - \MgP ) \times \Delta_P$ and thus it determines a class $[C] \in \MU{g}{P}$. Consider a metric on $\MU{g}{P}$ induced by Fenchel-Nielsen coordinates. The normal bundle can be used to induce a small tubular neighborhood of the boundary of $\bl_{\bar{F}} (\MU{g}{P} \times \Delta_P)$ so that each normal ray corresponds with a geodesic ray isometric to $[0,\rho)$. Taking $0<\epsilon<\rho$ this defines a family $\{[(C_\epsilon,\lambda_\epsilon)]\}_\epsilon$ of decorated (non-singular) surfaces in $\MgP \times \Delta_P$ with $\lim_{\epsilon \to 0} [C_\epsilon] = [C]$.

Now define a unital pair $(\lambda, \ord)$ by induction on the order. Recall that $\{ C_i \}_{i \in I}$ is the set of connected components of the normalization $\hat{C}$. This comes with an involution $\iota: N \to N$ and set $P_i = C_i \cap P$, $N_i = C_i \cap N$. Let $\{ C_i \}_{i \in I_0 \subset I}$ be the set of all irreducible components with $P_i \ne \varnothing$ and $\lim_{\epsilon \to 0} \lambda_\epsilon (p) >0$ for at least one $p \in P_i$. Set $\lambda(p) = \lim_{\epsilon \to 0} \lambda_\epsilon (p)$ for $p \in P_i$ and $\lambda(p) = 0$ for $p \in N_i$. Also let $\ord | \pi_0 (\amalg  C_i) \equiv 0$ for $i \in I_0$. This defines $(\lambda,\ord)$ at order zero. Notice that $\sum \lambda(p) = 1$ where the sum runs over all labeled points on components of order zero.

Lets assume now that $(\lambda,\ord)$ has been defined up to order $k$ and it is unital up to that order. We require a condition on degenerating geodesics. Either a few geodesics have been turned into semistable spheres (by being collapsed and becoming components of higher order) or they have given rise to decorations on nodes of components of order greater than zero (this will be made explicit on the inductive step). In the first case it is assumed that such semistable spheres are associated to components of order less than or equal to $k$. If $g^\epsilon_\beta$ is a geodesic being collapsed to a node $n$ we can express this as the limit $\lim_{\epsilon \to 0} g^\epsilon_\beta =n$. Let $l(g^\epsilon_\beta)$ be the length of such geodesic. Obviously $\lim_{\epsilon \to 0} l(g^\epsilon_\beta) =0$. Consider the limits \[ d(p_\alpha) = \lim_{\epsilon \to 0} \frac{\lambda_\epsilon (p_\alpha)}{\sum \lambda_\epsilon(p_\bullet) + \sum l(g^\epsilon_\bullet)}, \quad d(n) = \lim_{\epsilon \to 0} \frac{l (g^\epsilon_\beta)}{\sum \lambda_\epsilon(p_\bullet) + \sum l(g^\epsilon_\bullet)} \] where the first sum of the denominator runs over all labeled points $p_\bullet \in C_i $ with $i \in I - (I_0 \amalg I_1 \amalg \cdots \amalg I_k)$ and the second sum runs over all geodesics being collapsed to a node singularity that have not been turned into semistable spheres or have given rise to decorations on nodes.

Suppose $d(n)>0$ for some node $n$. Assume first that $n$ separates two components whose orders have already been assigned and thus are less than or equal to $k$. In this case we cut the surface along the node and glue a semistable sphere in between. We call this sphere $C_i$. Here $i=|I|+1$ and we have to include this number in the set $I$. If $q_1$, $q_2$ are the two elements of $N_i$ define \[ \lambda (q_1) = \lambda (q_2) = d(n)/2, \quad \lambda (p_1) = \lambda(p_2) =0, \quad \ord(C_i)=k+1 \] where $p_1 = \iota (q_1)$, $p_2 = \iota(q_2)$ are given in the obvious way. Now suppose that $n$ separates two components one of which has already been assigned an order. Let the other one be $C_i$. If $q \in N_i$ and $p=\lambda(q)$ then define \[ \lambda (q) = d(n), \quad \lambda(p) = 0, \quad \ord C_i = k+1. \] For every component $C_i$ with at least one $p \in P_i$ such that $d(p)>0$ define \[ \lambda(p) = d(p) \,\,\,\, \text{for all $p \in P_i$}, \quad \ord C_i = k+1. \] This produces a unital pair $(\lambda,\ord)$ up to order $k+1$. Since the type of the surface is finite this process exhausts all components of the normalization of the surface giving them orders and possibly creating along the way semistable spheres. This completes the definition of $\varphi$.

The map $\varphi$ is surjective on $\MgP \times \Delta_P$ because it is defined as the identity there. Given a point $[(C,\lambda,\ord)]$ where $C$ is a singular surface, it is possible to construct a one parameter family $\{ [C_t] \}_{0<t<1}$ satisfying the following conditions:

\begin{itemize}
\item $[C_t] \in \MgP$ for $0<t<1$
\item Let $k_i$ be the order of the component of the node $n_i$ such that $\lambda(n_i)
>0$. Let $g_i(t)$ be the geodesic giving rise to $n_i$ in $C_t$. If $l$ denotes the length of a geodesic then \[ l(g_i(t)) = t^{k_i} 2 \lambda(n_i) \quad \text{or} \quad l(g_i(t)) = t^{k_i} \lambda(n_i) \] depending on whether $n_i$ belongs to a semistable sphere or not.
\item $\lim_{t \to 0} [C_t] = [C] \in \partial \MU{g}{P}$
\end{itemize}

Suppose that $p_i \in P$ lies on a component of order $k_i$ and define by $p_i(t)$ its corresponding labeled point in $[C_t]$ for $0<t<1$. Then letting \[ \lambda(p_i(t)) = t^{k_i} \lambda (p_i) \] defines a path $\alpha(t)$ in $\MDgP$. It can be checked then that \[ \lim_{t \to 0} \alpha(t) = [(C,\lambda, \ord)].\] Moreover, the preimage of $\alpha$ under $\varphi$ also defines a path in $\bl_{\bar{F}} (\MU{g}{P} \times \Delta_P)$ (it is the same path since this map is the identity on the interior of the moduli space). This limit also exists and defines a point $x=\lim_{t \to 0} \alpha(t)$ in $\bl_{\bar{F}} (\MU{g}{P} \times \Delta_P)$ which is a preimage of $[(C,\lambda,\ord)]$ under $\varphi$.

The map $\varphi$ is injective on $\MgP \times \Delta_P$ because it is defined as the identity there. Let $x_1$, $x_2$ belong to $\bl_{\bar{F}} (\MU{g}{P} \times \Delta_P) - \MgP \times \Delta_P$ so that $x_1 \ne x_2$ and let $\varphi(x_1)=[(C_1,\lambda_1, \ord_1)]$, $\varphi(x_2)=[(C_2,\lambda_2, \ord_2)]$. Now consider the following cases. If $x_1$ and $x_2$ were generated by blowing up along the locus of singular surfaces with $[C_1] \ne [C_2]$ then $\varphi(x_1) \ne \varphi(x_2)$. In case $[C_1] = [C_2]$ it could also happen that $x_1$ and $x_2$ where generated by blowing up along different strata of $\bar{F}$. This will give rise to different order functions and hence again $\varphi(x_1) \ne \varphi(x_2)$. In the last case, if $[C_1] = [C_2]$ and $\ord_1 = \ord_2$ it can be showed that all the parameters left to consider in the decoration (the perimeter function) completely parametrizes this part of the blowup and therefore $x_1 \ne x_2$ implies $\varphi(x_1) \ne \varphi(x_2)$.

Finally, the topology of $\MUD{g}{P}$ is induced from the topology of the blowup through this bijection. In $\MgP \times \Delta_P \subset \MUD{g}{P}$ it is the same topology as usual since the map is defined as the identity there.

\end{proof}

Notice that by definition $\bl_{\bar{F}} (\MU{g}{P} \times \Delta_P)$ and $\MUD{g}{P}$ are homeomorphic.

\begin{coro} There is an homeomorphism
\[ \MUD{g}{P} \cong \MU{g}{P} \times \Delta_P  \] and therefore $\MUD{g}{P}$ is Hausdorff and compact. \label{maincoro1}
\end{coro}

\begin{proof}
Since $\bl_{\bar{F}} (\MU{g}{P} \times \Delta_P) \cong \MU{g}{P} \times \Delta_P$, Lemma~\ref{busb} provides such homeomorphism.
\end{proof}

Now we turn our attention to another moduli space.

\begin{df}
The space $\MOD{g}{P}$ is obtained from $\MUD{g}{P}$ by forgetting the decorations by tangent directions. The canonical projection $\MUD{g}{P} \to \MOD{g}{P}$ induces then a quotient topology on $\MOD{g}{P}$.
\end{df}

As a result of the previous theorem the space $\MOD{g}{P}$ can be defined in two analogue ways. One can define decorated semi-stable surfaces together with a topology as before or one can blow up $\MO{g}{P} \times \Delta_P$. The second definition requires an extra step: to forget the decorations by tangent directions produced by the blowup. The topology is then the quotient topology induced by a similar projection as in the previous definition. 

\begin{thm}
There is a map $\MOD{g}{P} \to \MO{g}{P} \times \Delta_P$ which is a homeomorphism in the interior and $\MOD{g}{P}$ has conical singularities along the locus of singular surfaces. \label{mainthm1}
\end{thm}

\begin{proof}
From the previous definition $\MOD{g}{P} \cong \hat{B}$ where the space $\hat{B}$ is obtained from $\bl_{\bar{F}} (\MO{g}{P} \times \Delta_P)$ by forgetting the decorations by tangent directions and thus inherits the quotient topology. This real oriented blowup provides such map which is a homeomorphism on the interior by construction. For the second part consider two cases. If a stratum of $\bar{F}$ corresponds with singular surfaces where all components with labeled points are of order zero then the new boundary created and the subsequent quotient corresponds locally with the picture in Corollary~\ref{torichood} modulo some finite group action. Otherwise there are components of order greater than zero with labeled points having all decorations equal to zero. The former case then generalizes to take care of the latter and a neighborhood of a point in the blowup after taking the quotient will have a toric neighborhood.
\end{proof}

\begin{coro}
The space $\MOD{g}{P}$ is Hausdorff, compact and homotopic to $\MO{g}{P}$. \label{maincoro2}
\end{coro}

The following lemma will help us understand the examples.

\begin{lem}
The preimages of points in the strata of $\bar{F}$ in $\bl_{\bar{F}} (\MU{g}{P} \times \Delta_P)$ under the natural projection are products of simplices modulo finite groups.
\end{lem}

\begin{proof}
By definition of $\bar{F}$ a point in the filtration lies in the locus of multi-intersecting strata. This gives the topological type, decoration by tangent directions, and conformal structure on the irreducible components of the normalization. The extra information introduced by the blowup is the half sphere as in the proof of Lemma~\ref{busb}. If a metric is given this induces a metric on the normal bundle. This metric then can be used to give the half sphere the desired parametrization by a product of blown-up simplices, one for every order of the surface, modulo a finite group action.
\end{proof}

\begin{figure}
\includegraphics[width=9cm]{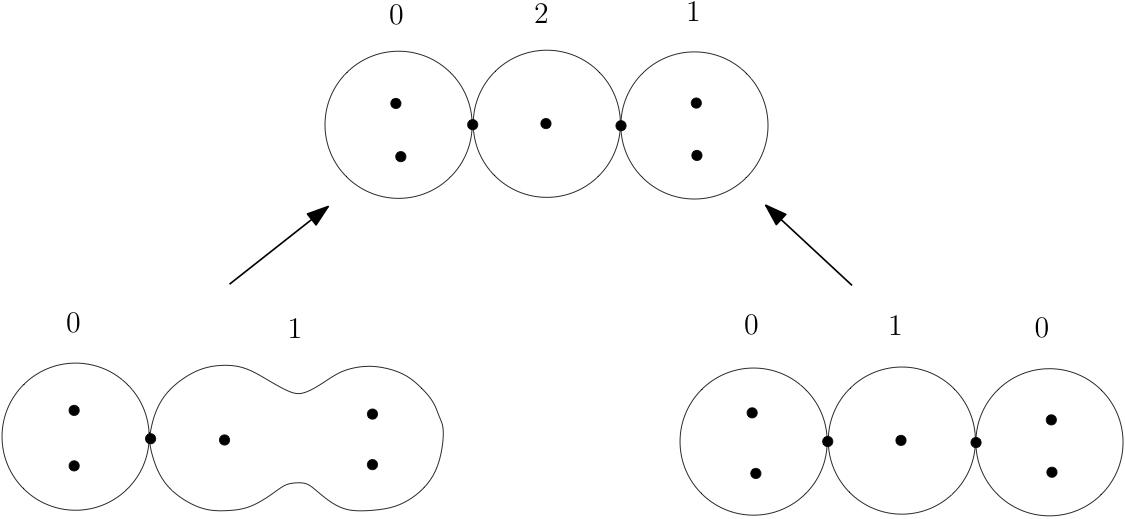}
\caption{The topological type of a surface arising as the intersection of strata (notice how the orders are added).} \label{surftrio}
\end{figure}

\begin{ex}
On Figure~\ref{surftrio} we can see how the topological type and combinatorics of the order can be determined by the intersection of the strata being blown up in the definition of the decorated moduli space. The actual decorations of the surface in the middle is determined by a point in $\Delta^1 \times \Delta^2 \times \Delta^1$ corresponding with the components of order 0, 2, and 1 respectively. Figure~\ref{busas} shows surfaces whose associated closure of blown-up simplex parameterizing the decorations is given on Figure~\ref{bu}. The extra faces induced by the blowup (and captured by the closure) correspond with decorations on components of higher order.
\end{ex}

\begin{figure}
\includegraphics[width=7cm]{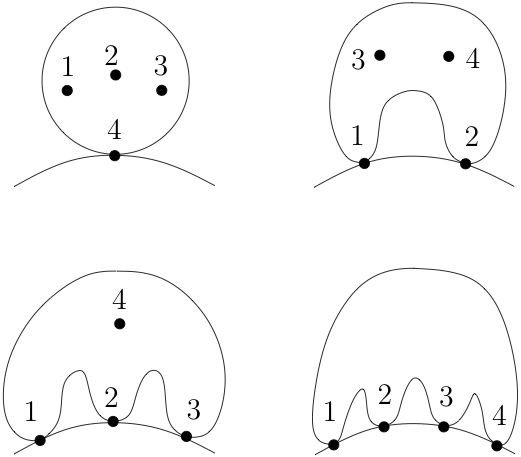}
\caption{Surfaces associated with the blown-up simplices on Figure~\ref{bu}.} \label{busas}
\end{figure}

\begin{rem}
By generalizing Figure~\ref{bu} and Figure~\ref{busas} one can obtain the cyclohedron from degeneration of surfaces. This is not the case with the associahedron. In the case of ${\mathcal K}^3$ this is because all possible three dimensional simplices arising from degeneration of surfaces are illustrated in those figures and ${\mathcal K}^3$ is not among them. This also implies that it will not show up as a face in a higher dimensional blown-up simplex. To get the associahedron one needs to consider compactified versions of moduli of Riemann surfaces with boundary as in \cite{liu} or \cite{cos:dpv}. More recently in \cite{dhv} we can find a nice treatment of this connection between bordered Riemann surfaces and associahedral polytopes. 
\end{rem}

\begin{ex}
The space $\MU{0}{P}$ where $|P|=4$ can be identified with the Riemann sphere with three removed open disks corresponding with the three possible ways in which the Riemann sphere with four labeled points can degenerate. The space $\MUD{0}{P}$ is the union of $\M_{0,P} \times \Delta_P$ with three copies of the space $\S^1 \times T$ where $T$ is a three dimensional simplicial complex obtained from gluing three solids: two copies of $\Delta^2 \times \Delta^1$ and the real oriented blowup of $\Delta^3$ at two opposite edges corresponding to the decorations of the labeled points in each irreducible component of the stable surface. The interior of the blown-up copy of $\Delta^3$ corresponds with the first surface on Figure~\ref{surfcom}. The interior of the two copies of $\Delta^2 \times \Delta^1$ corresponds with the second surface on Figure~\ref{surfcom}. Finally, the intersection of these complexes are rectangles corresponding with the third surface on Figure~\ref{surfcom}.

\begin{figure}
\includegraphics[width=5cm]{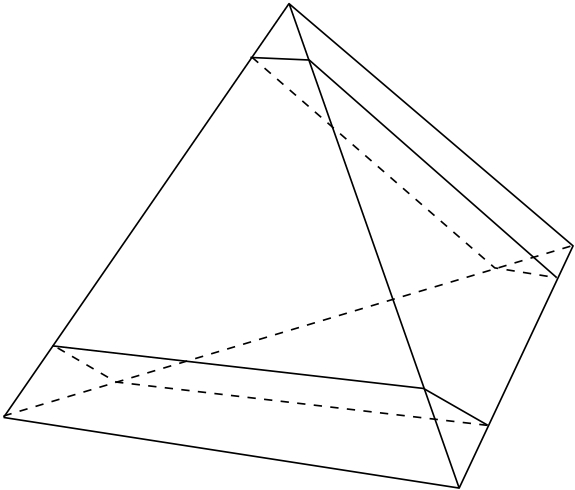}
\caption{The complex $T$.}
\end{figure}

A simple way to go from $\M_{0,P} \times \Delta_P$ to the boundary is to consider the geodesic $g(t)$ that is being collapsed and its length $l$. If $l(g(t)) \to 0$ and each resulting irreducible component contains a marked point with non-vanishing limit, then we land in the blown-up copy of $\Delta^3$ in $T$. If the decorations of both labeled points in a resulting irreducible component tend to zero then we land in one of the copies of $\Delta^2 \times \Delta^1$. To decide in which point we land let $d_1(t)$, $d_2(t)$ be such decorations and $n$ the decoration at the node. Then the decorations in the limit will be \begin{align*} 
d_1 &= \lim \frac{d_1(t)}{d_1(t) + d_2 (t) + l(g(t))} \\
d_2 &= \lim \frac{d_2(t)}{d_1(t) + d_2 (t) + l(g(t))} \\
n &= \lim \frac{l(g(t))}{d_1(t) + d_2 (t) + l(g(t))}.
\end{align*}

This works for $\MUD{0}{4}$ as well as $\MOD{0}{4}$. The only difference is that in the first case we keep track of the angles which give the decorations by tangent directions. Since the singular surfaces in $\MOD{0}{4}$ can only have one singularity the toric neighborhoods reduce to the cone over a circle which is homeomorphic to a disc.
\end{ex}

\begin{figure}
\includegraphics[width=12cm]{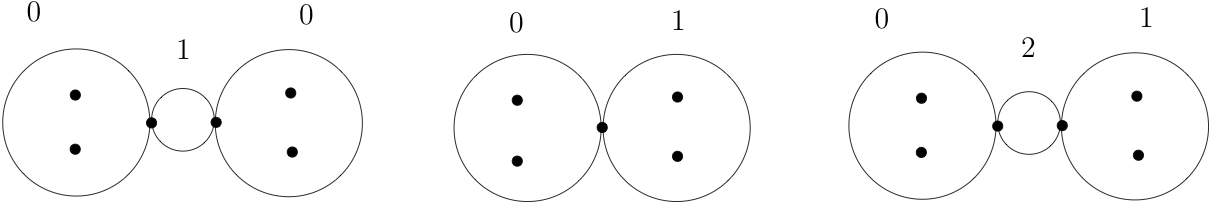}
\caption{Surfaces associated to the complex $T$.} \label{surfcom}
\end{figure}

\section{Semistable Ribbon Graphs}

\subsection{Ribbon Graphs}

By a graph we mean a combinatorial object consisting of vertices, edges that split into half-edges and incidence relations. We avoid isolated vertices. This is the same as a one dimensional CW-complex up to cellular homeomorphism.

We will need to consider a special graph with only one edge and no vertices homeomorphic to $\S^1$. We call this a {\bf semistable circle}. The following definition of ribbon graph allows then for the possibility of having multiple connected components, some of them possibly being semistable circles.

\begin{df}
A {\bf ribbon graph} $\Gamma$ is a finite graph together with a cyclic ordering on each set of adjacent half-edges to every vertex. \par\medskip
\end{df}

If $H$ is the set of half-edges and $v$ is a vertex of $\Gamma$ let $H_v$ be the set of adjacent half-edges to this vertex. The {\bf valence} of a vertex is then $|H_v|$. A {\bf trivalent} graph is one for which all vertices have valence three. A cyclic ordering at a vertex $v$ is an ordering of $H_v$ up to cyclic permutation. Once a cyclic ordering of $H_v$ is chosen, a cyclic permutation of $H_v$ is defined (an element of $\FS_{H_v}$): it moves a half-edge to the next in the cyclic order. Define by $\sigma_0$ the element of $\FS_H$ which is the product of all the cyclic permutations at every vertex and let $\sigma_1$ be the involution in $\FS_H$ that interchanges the two half-edges on each edge of $\Gamma$. Notice that if $\sigma_0$ does not act on certain half-edges it is because those half-edges belong to semistable circles (semistable circles have no vertices). This combinatorial data completely defines the ribbon graph. 

To be more precise, given a finite set $H$ and permutations $\sigma_0, \sigma_1 \in \FS_H$ such that $\sigma_0$ is a product of cyclic permutations with disjoint support and $\sigma_1$ is an involution without fixed points, then we can construct a ribbon graph $\Gamma$. A vertex of $\Gamma$ is then given as an orbit of $\sigma_0$ on $H$, while an edge is then an orbit of $\sigma_1$ on $H$. The set of vertices may be identified with $V(\Gamma)=H/\sigma_0$ and the set of edges with $E(\Gamma)=H/\sigma_1$. Semistable circles correspond with pairs of half-edges in the orbit of $\sigma_1$ that are missed by the action of $\sigma_0$. 

Let $\sigma_\infty = \sigma_0^{-1} \sigma_1$. The orbits of $\sigma_\infty$ will be called {\bf cusps} and they form the set $C(\Gamma)= H/\sigma_\infty$. The half-edges in the orbit of a cusp forms a cyclically ordered set of half-edges called a {\bf boundary cycle}. The obvious graph associated to the boundary cycle is called a {\bf boundary subgraph}. The reason for such terms will become evident later. For a semistable circle we let $\sigma_\infty$ be the identity. This implies that semistable circles have exactly two boundary cycles (each one consisting of only one half-edge). The cusps and the vertices of valence one or two will be called {\bf distinguished points}. Notice also that knowing $\sigma_1$ and $\sigma_\infty$ completely determines the ribbon graph structure since $\sigma_0= \sigma_1 \sigma_\infty^{-1}$.

A {\bf loop} is an edge incident to only one vertex and a {\bf tree} is a connected graph $T$ satisfying $\bar{H}_*(T)=0$. 

An {\bf isomorphism} of ribbon graphs is a graph isomorphism preserving the cyclic orders on each vertex. Therefore, two graphs $\Gamma$, $\Gamma'$ are isomorphic when there is a bijection $\eta: H \to H'$ between the set of half-edges of these two graphs that commutes with $\sigma_0$, $\sigma_0'$ and $\sigma_1$, $\sigma_1'$. In particular this implies that the boundary cycles are preserved, \emph{i.e.} $\eta$ also commutes with $\sigma_\infty$, $\sigma_\infty'$.  If we restrict to automorphisms of a graph it is clear that this will generate a group with this definition. The group of automorphisms of the semistable circle is $\Z/ 2 \Z$.

\begin{ex}
Consider the ribbon graph in Figure~\ref{theta}. Denote by $h_i$ the half-edges of the graph as in the Figure and let the cyclic ordering be induced by the counter-clockwise orientation. Then \begin{align*}
\sigma_0 &= (h_1 h_5 h_3)(h_2 h_6 h_4) \\
\sigma_1 &= (h_1 h_2)(h_3 h_4)(h_5 h_6) \\
\sigma_\infty &= (h_1 h_4 h_5 h_2 h_3 h_6)
\end{align*}

Its group of automorphisms is $\Z/2\Z \times \Z/3\Z$ where the $\Z/ 2 \Z$ factor is induced by $\sigma_1$.

\begin{figure}
\includegraphics[width=3.8cm]{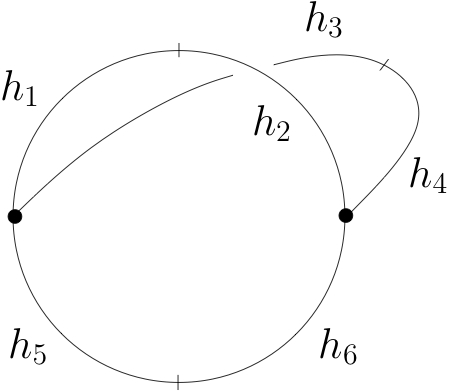}
\caption{A ribbon graph with both vertices having the counter-clockwise orientation.} \label{theta}
\end{figure}

\end{ex}

An interesting construction associated to a ribbon graph is its {\bf dual graph}; it is obtained by passing from $(H_\Gamma; \sigma_0,\sigma_1)$ to $(H_\Gamma; \sigma_\infty,\sigma_1)$. This new ribbon graph will be denoted by $\Gamma^*$. Notice that there is a natural identification between the sets $E(\Gamma)$ and $E(\Gamma^*)$. The dual graph of the semistable circle is itself.

From now on all figures of ribbon graphs will have the cyclic ordering induced by the counter-clockwise orientation.

\begin{rem}
The set of half-edges can be identified with the set of oriented edges in two ways. To each oriented edge we can assign the source or target half-edge. We use the one assigning the source. The involution $\sigma_1$ switches the orientation of every edge.
\end{rem}

To every ribbon graph $\Gamma$ we can associate an oriented surface $\surf(\Gamma)$ constructed as follows. To each oriented edge $e$ we can associate a semi-infinite rectangle $K_e=|e|\times \R_{\ge 0}$ at the base where $|e|$ is homeomorphic to the closed unit interval. Let $\ol{K}_e$ be its one-point compactification. Now identify the base of $\ol{K}_e$ with the base of $\ol{K}_{\sigma_1(e)}$ and the right-hand edge of $\ol{K}_e$ with the left-hand edge of $\ol{K}_{\sigma_\infty(e)}$. There are some special points coming from the compactification (after adding them into the surface), they can be identified with the orbits of $\sigma_\infty$, and that's why we call them cusps. Each connected component of the graph has genus $g_i=(2-\chi_i-n_i)/2$ where $\chi_i=|V(\Gamma_i)|-|E(\Gamma_i)|$, $\Gamma_i$ is the $i$-th connected component of $\Gamma$ and $n_i$ is the number of cusps in that component. The surface comes with a natural orientation given by the tiles since they are naturally oriented and their orientations match each other because of the way we glued them.

This construction can also be applied to semistable circles. Even though semi-stable circles have no vertices they still have half-edges and thus they also have two orientations corresponding to their boundary cycles. We glue the semi-infinite rectangles in order to obtain an infinite cylinder with two cusps. One may worry that since there is no vertex there is no way to know where to start gluing the rectangle. However, the choice of a base point becomes irrelevant because the moduli of semistable spheres is trivial.

There is also a natural identification between $\surf(\Gamma)$ and $\surf(\Gamma^*)$ where $\Gamma^*$ is the dual graph.

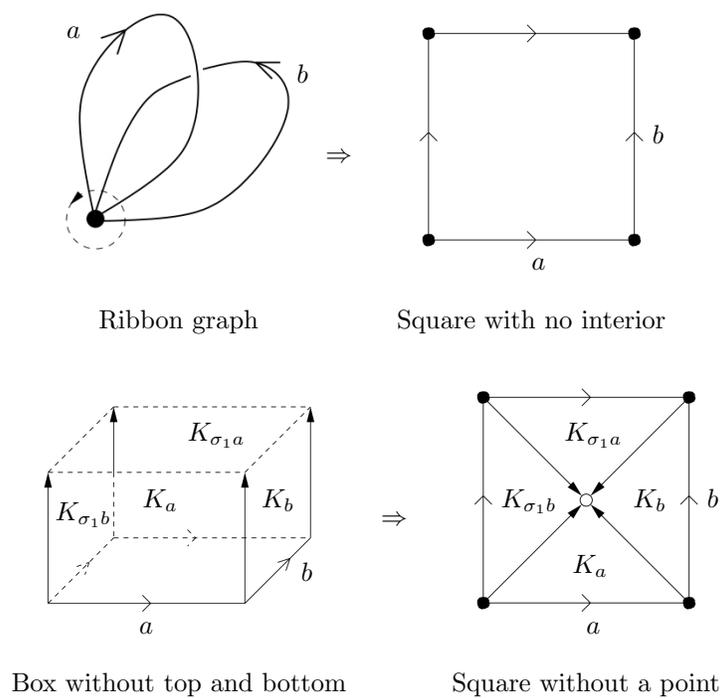
\begin{figure}
$$
\begin{array}{ccc}
\includegraphics[width=3cm]{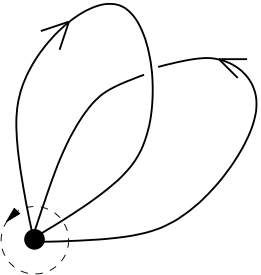} \put(-85,80){$a$} \put(2,63){$b$} & \put(3,33){$\Rightarrow $} \qquad & \includegraphics[width=3cm]{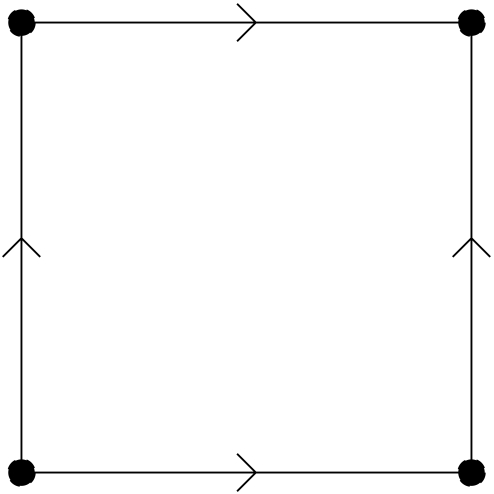} \put(3,40){$b$} \\
 & & \put(0,10){$a$} \\
\text{Ribbon graph} & & \text{Square with no interior}
\end{array}
$$

\bigskip

$$
\begin{array}{ccc}
\includegraphics[width=3.6cm]{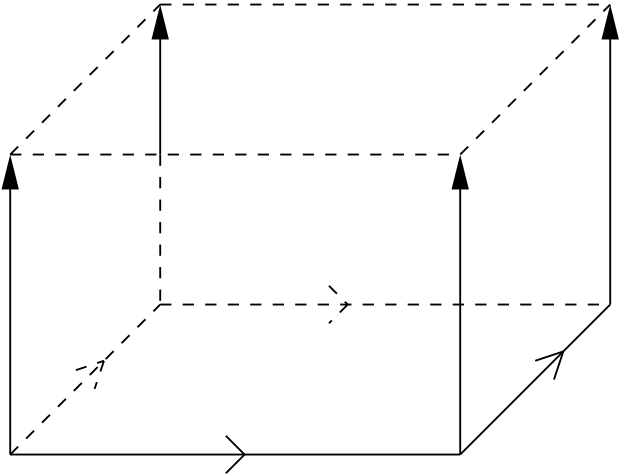} \put(-5,12){$b$} \put(-98,35){$K_{\sigma_1b}$} \put(-20,40){$K_b $} \put(-65,40){$K_a$} \put(-48,65){$K_{\sigma_1a}$} & \put(3,33){$\Rightarrow$} \qquad & \includegraphics[width=3cm]{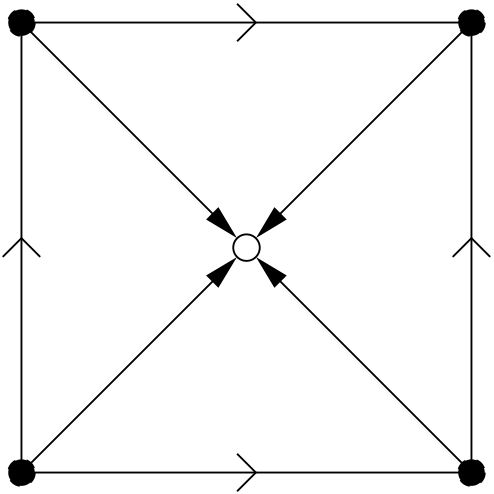} \put(3,40){$b$} \put(-75,40){$K_{\sigma_1b}$} \put(-25,40){$K_b$} \put(-48,15){$K_a$} \put(-51,65){$K_{\sigma_1a}$} \\
 \put(-15,10){$a$} & & \put(0,10){$a$} \\
\text{Box without top and bottom} & & \text{Square without a point}
\end{array}
$$

\bigskip

\includegraphics[width=7cm]{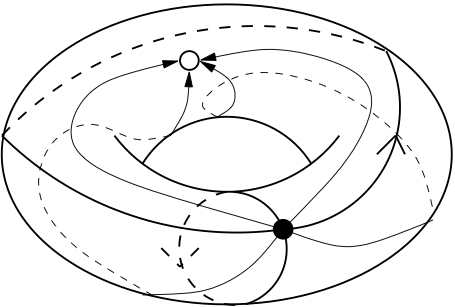}
\caption[Once-punctured torus.]{Once-punctured torus, adding the puncture gives $\surf(\Gamma)$.}
\end{figure}

\begin{df}
A {\bf $P$-labeled ribbon graph} is a ribbon graph together with an injection $x: P \hookrightarrow V(\Gamma) \sqcup C(\Gamma)$ whose image contains all distinguished points. The elements of the image will be called {\bf labeled points}. 
\end{df}

An {\bf isomorphism} of $P$-labeled ribbon graphs is a ribbon graph isomorphism that preserves the labels. In particular, the automorphism group of the semistable circle is trivial.

\begin{df}
The {\bf Euler characteristic} of a $P$-labeled ribbon graph is defined as the Euler characteristic of the graph minus $|P|$. The semistable circle is defined to have Euler characteristic equal to zero.
\end{df}

\begin{rem}
Clearly, if $\Gamma$ is a $P$-labeled ribbon graph then $\surf(\Gamma)$ inherits a $P$-labeling in the form of a function $x: P \hookrightarrow \surf(\Gamma)$. The topological type $(g,|P|)$ of a $P$-labeled ribbon graph refers to the genus $g$ of the generated surface and the number $P$ of labels. It is also easy to check that the Euler characteristic of the ribbon graph is the same as the Euler characteristic of the surface associated to it.
\end{rem}

\subsection{Gluing Construction}

Fix a vertex $v$ in a ribbon graph. We can construct a new ribbon graph by replacing $v$ with $|H_v|$ edges and $|H_v|$ vertices as in Figure~\ref{blow-upv}. The new ribbon graph is the {\bf blowup} of $v$. This operation adds one extra boundary cycle to the ribbon graph.

\begin{figure}
\includegraphics[width=6.5cm]{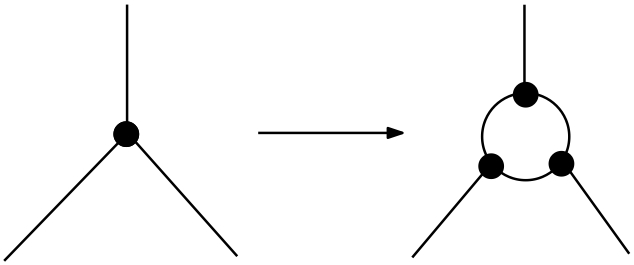}
\caption{How to blow up a vertex.} \label{blow-upv}
\end{figure}

A boundary cycle is called {\bf injective} if any two half-edges in this orbit are not in the same orbit of $\sigma_0$ or $\sigma_1$. This implies that the boundary subgraph is homeomorphic to a circle. For example, the extra boundary cycle generated in the blowup is always injective.

By {\bf disjoint} boundary cycles we mean boundary cycles that do not share any half edges in the same orbit of $\sigma_0$ or $\sigma_1$. This means that the associated boundary subgraphs do not intersect. Given two disjoint boundary cycles with at least one of them being injective we can produce a finite family of ribbon graphs as follows.

Since both boundary cycles correspond with subgraphs that can be identified with CW-complexes themselves choose parametrizations of each subgraph by $\S^1$. The parametrization of the subgraph associated to the injective boundary cycle must be compatible with the natural counter-clockwise orientation of $\S^1 \subset \C$, \emph{i.e.} it follows the cyclic order of the boundary cycle. The other subgraph is parametrized with the opposite orientation. 

\begin{figure}
\includegraphics[width=11cm]{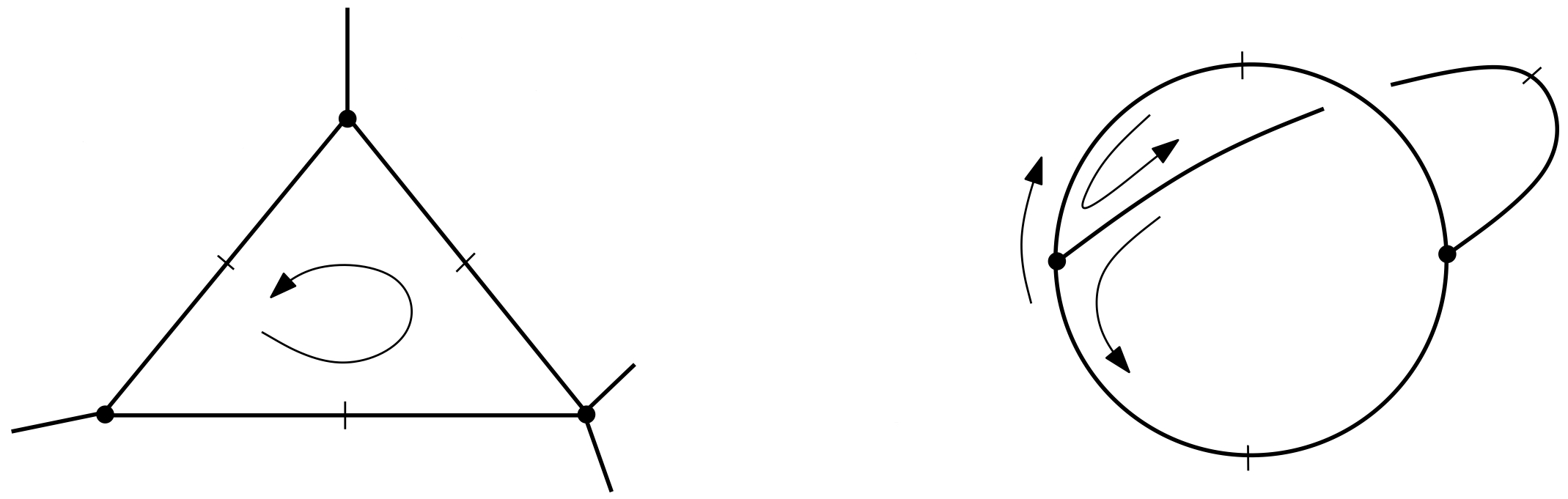}
\caption[Positive and negative boundary cycles.]{Two boundary cycles, the one on the left is injective, the one on the right is not.} \label{firstlast}
\end{figure}

Now we glue both subgraphs via the map identifying two points if their preimages under the parametrization coincides. This gives an obvious new set of half-edges and vertices and it can be shown that the resulting graph is a ribbon graph (this is the reason why we introduced disjointness and injectivity of boundary cycles). Now discard the parametrization left in the gluing. This results then in a ribbon graph called a {\bf gluing}. There is also a way to define this gluing construction in a purely combinatorial way but it lacks the geometrical intuition.

To produce a family of ribbon graphs change the parametrizations and keep only one representative from each isomorphism class of ribbon graph thus created. Since there is a bound on the size of the resulting graphs and these graphs are also finite there will be only a finite number of isomorphism classes. 

\begin{df}
Given a vertex and a boundary cycle whose associated graph does not include the given vertex we define a {\bf gluing} by applying the gluing construction to the blowup of the vertex and the given boundary cycle. \label{gluing}
\end{df}

This construction is well defined because the blowup is injective and the given condition implies that the boundary cycles are disjoint.

The previous definition is a sort of ``desingularization" of graphs.

\begin{figure}
\includegraphics[width=12.5cm]{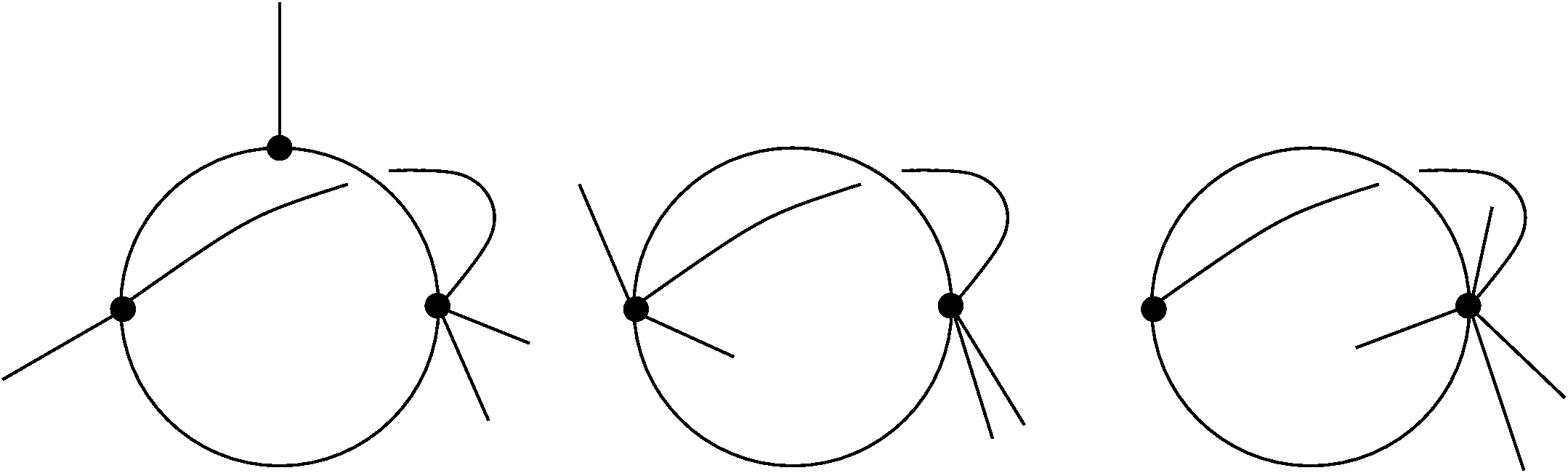}
\caption[Different gluings]{Different gluings of two boundary cycles from Figure~\ref{firstlast}.}
\end{figure}

\subsection{Semistable Ribbon Graphs}

Let us describe two ribbon graphs we can obtain from a proper subset of edges $Z \subset E(\Gamma)$. One will be associated to $Z$ and the other to its complement in $E(\Gamma)$. Denote by $\Gamma_Z$ the subgraph with set of edges $Z$ and $H_Z$ its set of half-edges. The ribbon graph structure is induced by $\sigma_0$ and $\sigma_1$ in the following way. The new $\sigma^{\Gamma_Z}_1$ is just the restriction while $\sigma^{\Gamma_Z}_0$ is defined by declaring $\sigma^{\Gamma_Z}_0(h)$, with $h \in H_Z$, to be the first term in the sequence $(\sigma_0^k(h))_{k>0}$ that is in $H_Z$.

The proper subset $Z \subset E(\Gamma)$ of edges of a ribbon graph induces a ribbon graph structure on the graph determined by the complement of $Z$ in $E(\Gamma)$. We will denote this graph by $\Gamma / \Gamma_Z$. The new graph has set of edges $E(\Gamma) - Z$ with induced set of half-edges $H_{\Gamma/\Gamma_Z}$. Since $\sigma_1$ and $\sigma_\infty$ completely determine the ribbon graph structure it is enough to define them in $H_{\Gamma/\Gamma_Z}$. The new involution is just the restriction $\sigma^{\Gamma/\Gamma_Z}_1= \sigma_1 |_{H_{\Gamma/\Gamma_Z}}$. Given $h \in H_{\Gamma/\Gamma_Z}$ we define $\sigma^{\Gamma/\Gamma_Z}_\infty (h)$ to be the first term of the sequence $(\sigma_\infty^k(h))_{k>0}$ that is in $H_{\Gamma/\Gamma_Z}$.

\begin{rem}
If $\Gamma_Z$ is simply connected then $\Gamma / \Gamma_Z$ is topologically the result of collapsing each component of $\Gamma_Z$ to a point. In general this is not a topological quotient. Figure~\ref{nottopquo} shows an example of this last case. It turns out that this definition allows us to track the creation of nodes at the graph level.
\end{rem}

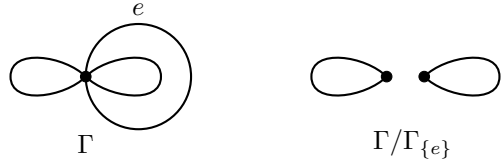
\begin{figure} 
\begin{tikzpicture}
\draw[thick] (0,0) to [out=135, in=90] (-1,0) to [out=270,in=225] (0,0);
\draw[thick, rotate=180] (0,0) to [out=135, in=90] (-1,0) to [out=270,in=225] (0,0);
\draw[thick] (0.7,0) circle (0.7);
\draw[fill=black] (0,0) circle (2pt);
\node at (0.7,0.9) {$e$};
\node at (0,-0.9) {$\Gamma$};

\draw[thick, xshift=4cm] (0,0) to [out=135, in=90] (-1,0) to [out=270,in=225] (0,0);
\draw[fill=black, xshift=4cm] (0,0) circle (2pt);
\draw[thick, rotate=180, xshift=-4.5cm] (0,0) to [out=135, in=90] (-1,0) to [out=270,in=225] (0,0);
\draw[fill=black, xshift=4.5cm] (0,0) circle (2pt);
\node at (4.35,-0.9) {$\Gamma/\Gamma_{\{ e \}}$};
\end{tikzpicture}
\caption[Not a topological quotient]{The original graph has one vertex while the second one has two and it is disconnected.} \label{nottopquo}
\end{figure}

We now describe how to collapse edges in a $P$-labeled ribbon graph without changing the homeomorphism type of $\surf(\Gamma)$ relative to $P$.

\begin{df}
A subset $Z \subset E(\Gamma)$ of a $P$-labeled ribbon graph $\Gamma$ is called {\bf negligible} if each connected component of $\Gamma_Z$ is either a tree with at most one labeled point or a homotopy circle without labeled points that contains a boundary subgraph.
\end{df}

\begin{df}
If $\Gamma$ is a $P$-labeled ribbon graph and $Z \subset E(\Gamma)$ is a negligible subset define the {\bf edge collapse} of $\Gamma$ respect to $Z$ as $\Gamma/Z = \Gamma/\Gamma_Z$ with the induced $P$-labeling. \label{edgecollap}
\end{df}

\begin{rem}
Collapsing a tree with at most one labeled point does not change the injectivity of the labels. Collapsing a homotopy circle without labeled points that contains a boundary subgraph is called a {\bf total collapse} and in this case the label of the corresponding cusp turns into a label of the induced vertex. The injectivity of this labeling is still preserved.
\end{rem}

\begin{lem}
If $Z$ is negligible then $\surf(\Gamma) \cong  \surf(\Gamma/Z)$ relative to $P$.
\end{lem}

\begin{proof}
It is possible to exhibit a sequence of homeomorphisms starting at $\surf(\Gamma)$ and ending at $\surf(\Gamma/Z)$. If a connected component of $\Gamma_Z$ is a tree with at most one labeled point let $e$ be an edge in that tree. As $e$ is contracted the result on the associated surface is to contract $K_e$ to an interval (one vertex goes to one vertex of the interval and the opposite edge to this vertex is contracted to the other vertex of the interval). This can be done to all edges of the tree without changing the injectivity of the labels. The same can be done on a homotopy circle without labeled points that contains a boundary subgraph. The difference is that in the last step we have a loop being contracted to a point labeling the resulting vertex. This collapse also respects the injectivity of the labels because the homotopy circle did not have a labeled point on it. Such process does not change the homeomorphism type of the surface.
\end{proof}

We can also collapse more general graphs allowing only mild degenerations. If we collapse more arbitrary subsets of edges the homeomorphism type is not preserved but we can show that the singularities thus obtained are simple. We start with the following definition.

\begin{df}
A proper set of edges $Z$ is {\bf semistable} if no component of $\Gamma_Z$ is the set of edges of a negligible subset and every univalent vertex of $\Gamma_Z$ is labeled.  
\end{df}

\begin{rem} \label{trees}
If $Z$ is semistable then every contractible component of $\Gamma_Z$ contains at least two labeled points (otherwise it would be negligible). A component that is a homotopy circle without labeled vertices is necessarily a topological circle because univalent vertices must be labeled. It is also not a boundary subgraph of $\Gamma$ or else it would be negligible.
\end{rem}

\begin{lem}
Given a ribbon graph $\Gamma$ every proper subset $Z \subset E(\Gamma)$ contains a unique maximal semistable subset $Z^{sst} \subset Z$. \label{maxsemis}
\end{lem}

\begin{proof}
We give an algorithm to find $Z^{sst}$. Starting from $Z$ remove all edges containing an unlabeled vertex of valence one. Repeat this process until we can not delete further edges. All remaining univalent vertices are labeled. Now throw away all boundary subgraphs with no labeled vertices. At the end what remains is $Z^{sst}$ and we have  $Z^{sst} = \varnothing$ if and only if $Z$ is negligible. The uniqueness of $Z^{sst}$ follows from construction.
\end{proof}

\begin{df}
Let $Z$ be a semistable subset of $\Gamma$. The {\bf reduction} of $\Gamma_Z$ is the result of deleting unlabeled vertices of valence two. We denote the reduction by $\hat{\Gamma}_Z$. 
\end{df}

The reduction of a homotopy circle with no labeled vertices corresponding with a semistable subset is in fact a semistable circle.

A semistable subset $Z$ is {\bf stable} if every component of $\Gamma_Z$ that is a topological circle contains a labeled vertex. An arbitrary proper subset $Z$ contains a unique maximal stable subset $Z^{st}$, which is obtained from $Z^{sst}$ by getting rid of the components that are topological circles without labeled vertices.

\begin{df}
If $\Gamma$ is a $P$-labeled ribbon graph and $Z \subset E(\Gamma)$ is an arbitrary subset define the {\bf edge collapse} of $\Gamma$ respect to $Z$ as the disjoint union \[ \Gamma/Z = \Gamma/\Gamma_Z  \sqcup \hat{\Gamma}_{Z^{sst}} \] with the induced $P$-labeling.
\end{df}

\begin{rem}
This generalizes Definition~\ref{edgecollap} because when $Z$ is negligible $Z^{sst} = \varnothing$ by Lemma~\ref{maxsemis}. Also notice that if $Z_1 \cup Z_2 = \varnothing$ then $\Gamma / (Z_1 \sqcup Z_2) = (\Gamma/Z_1)/Z_2$. \label{splitcollapse}
\end{rem}

The next step is to introduce a generalization of ribbon graphs that will give a cellular decomposition of the decorated moduli space of semistable Riemann surfaces. This is similar to Looijenga's definition in \cite[9.1]{loo} but some changes were required.

Let $Z$ be semistable. Take a vertex in $\Gamma/\Gamma_Z$. This is represented by an orbit of $\sigma^{\Gamma/\Gamma_Z}_0$. If any of the elements in that orbit is the image under $\sigma_0$ of an element of $H_Z$ we call that vertex {\bf exceptional}. In that case there is a corresponding orbit of $\sigma^{\Gamma_Z}_\infty$  that is not an orbit of $\sigma_\infty$ and such that the orbit of the exceptional vertex under $\sigma_0$ has non-trivial intersection with that particular orbit of $\sigma_\infty^{\Gamma_Z}$. In this case we call the elements of the corresponding orbit of $\sigma^{\Gamma_Z}_\infty$ an {\bf exceptional boundary cycle} and the associated subgraph an {\bf exceptional boundary subgraph}.

Consider an involution without fixed points $\iota$ on a subset $N \subset V(\Gamma) \sqcup C(\Gamma)$. The elements of $N$ will be called {\bf nodes}, two elements of the same orbit are {\bf associated} and in this case we may also say that the corresponding connected components of the graph are associated. {\bf Cusp-nodes} and {\bf vertex-nodes} are defined in an obvious way. This involution allows us to identify points in $\surf(\Gamma)$. Denote by $\surf(\Gamma,\iota)$ the resulting surface. Let $\Gamma = \cup_{i \in I} \Gamma_i$ where the $\Gamma_i$'s are the connected components of $\Gamma$. Thus $\pi_0(\Gamma) = \{ [\Gamma_i] \} \cong I$. Set $V_i=V(\Gamma_i)$, $C_i=C(\Gamma_i)$, and $N'_i = N(\Gamma_i)$ the nodes in the $i^{\text{th}}$ component of $\Gamma$. We will only consider graphs with involutions for which the following properties apply:

\begin{enumerate}
\item A connected component of the graph cannot be associated to itself.
\item Two semistable circles cannot be associated.
\item The two cusps of a semistable circle are nodes.
\item A cusp-node can only be associated to a vertex-node and vice versa.
\item The surface $\surf(\Gamma,\iota)$ must be connected.
\end{enumerate}

An {\bf order} for $\Gamma$ is a function $\ord: \pi_0 (\Gamma) \to \N$ satisfying the following properties.
\begin{enumerate}
\item[(i)] If $\ord([\Gamma_i])=k>0$ then there exist $j$ such that $\ord([\Gamma_j])=k-1$.
\item[(ii)] Let $p \in N'_i$ and $q=\iota(p) \in N'_j$, then $p \in C_i$ if and only if  $q \in V_j$ by property (4) on the previous list. In this case we require that $\ord([\Gamma_j])<\ord([\Gamma_i])$.
\end{enumerate}

The following gives some insight into this definition and is not hard to prove.

\begin{lem} \label{order}
Given an order $\ord: \pi_0 \Gamma \to \N$ we have:

\begin{enumerate}
\item If $p\in N'_i$ and $\ord([\Gamma_i])=0$ then $p \in V$.
\item There is a constant $m \in \N$ such that $\ord([\Gamma_i])\le m$ for all $i$ and given $k$ such that $0 \le k \le m$ there exist $i$ with $\ord([\Gamma_i])=k$.
\end{enumerate}
\end{lem}

\begin{df}
A {\bf semistable ribbon graph} is a ribbon graph $\Gamma$ together with an involution $\iota$ as above and an order function $\ord$.
\end{df}

\begin{rem}
A ribbon graph can be viewed as a semistable ribbon graph with $N=\varnothing$. Notice also that $\surf(\Gamma)$ is the normalization of $\surf(\Gamma,\iota)$. When $N \ne \varnothing$ we call the graph {\bf singular}.
\end{rem}

\begin{figure}
\includegraphics[width=11.5cm]{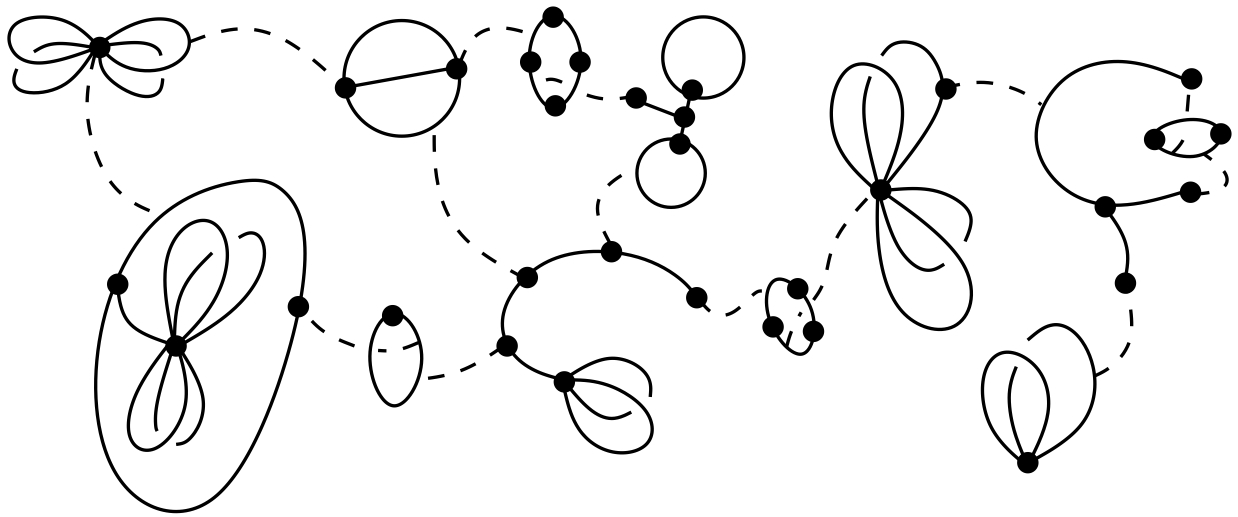}
\caption[Semistable ribbon graph.]{Semistable ribbon graph whose associated surface is isomorphic to the one in Figure~\ref{surford}.} \label{strig}
\end{figure}

\begin{df}
A {\bf $P$-labeled semistable ribbon graph} is a semistable ribbon graph together with an inclusion $x:P \hookrightarrow V(\Gamma) \sqcup C(\Gamma)$ satisfying:
\begin{enumerate}
\item The image $x(P)$ is disjoint from the set of nodes.
\item The union $x(P) \cup N$ contains all distinguished points.
\end{enumerate}
This inclusion is called a $P$-labeling. An {\bf isomorphism} in this case is an isomorphism of the underlying ribbon graph respecting the involution and order as well as the labeling.
\end{df}

A topological surface satisfying all the properties of a $P$-labeled semistable Riemann surface except for its complex structure and the exact value of the positive decorations by real numbers is called a {\bf $P$-labeled semistable topological surface}. This means we remember the order function and whether a decoration is zero or non-zero.

\begin{lem}
If $\Gamma$ is a $P$-labeled semistable ribbon graph then $\surf(\Gamma)$ is a $P$-labeled semistable topological surface.
\end{lem}

\begin{proof}
We need to show that every component of the normalization of $\surf(\Gamma)$ has non-positive Euler characteristic, \emph{i.e.} the Euler characteristic of the components of $\surf(\Gamma) - (N \sqcup x(P))$. We know that $\surf(\Gamma) - C(\Gamma)$ admits $\Gamma$ as a deformation retract. If a component is contractible then it must have at least two labeled points or nodes because such graph has at least two univalent vertices and the union $x(P) \cup N$ contains all distinguished points. This makes the Euler characteristic negative on those components. If the component is a topological circle the Euler characteristic is at most zero. In any other case the connected component of the graph will have negative Euler characteristic.
\end{proof}

We are almost ready to define the edge collapse for semistable ribbon graphs. The order function keeps track of how the graph degenerates and to satisfy its definition we are not allowed to collapse all the edges associated to all components of a given order. Otherwise there would be a ``gap'' in the order function (we would be missing a number in the list of orders in contradiction with Lemma~\ref{order}). This is why we have the following definition. A subset of edges of a given $P$-labeled semistable ribbon graph is called {\bf collapsible} if it does not contain the set of edges of the union of all components of a fixed order for any order $k$. For metric ribbon graphs this type of collapse will be avoided naturally because the metrics considered are unital.

The definition of negligible subset needs to be modified for $P$-labeled semistable ribbon graphs. A boundary subgraph is negligible even if it corresponds to a cusp-node. In this case a total collapse induces an involution without fixed points that would associated a vertex-node with the newly generated vertex-node. This is in contradiction of the definition of semistable ribbon graph. To fix this we simply \emph{exclude} from the definition of {\bf negligible} subset all those components that are homotopy circles without labeled points that contain a boundary subgraph giving rise to a cusp-node. Notice this only makes sense when we have the $P$-labeling and the semistable ribbon graph structure (that includes the involution). The main consequence is that now when doing a total collapse of a boundary subgraph corresponding to a cusp-node this subgraph will not simply disappear. Instead, it will generate a semistable circle. The induced involution without fixed points will associate the old vertex-node and the newly generated vertex-node to both cusps of this semistable circle. In this way the induced involution satisfies the condition of only associating cusp-nodes with vertex-nodes and vice versa. Another consequence is that a {\bf semistable} subset of a $P$-labeled semistable ribbon graph could possibly contain boundary subgraphs giving rise to cusp-nodes.

\begin{df}
If $\Gamma$ is a $P$-labeled semistable ribbon graph and $Z\subset E(\Gamma)$ is a collapsible subset of edges, the {\bf edge collapse} is a new $P$-labeled semistable ribbon graph defined as follows.
\begin{itemize}

\item As a $P$-labeled ribbon graph the edge collapse is $\Gamma/Z$. Notice that the change on the definition of negligible subset creates semistable circles for each total collapse of a homotopy circle without labeled points that corresponds to a cusp-node.

\item There is a new order function defined inductively. For this we express $Z$ as a disjoint union $Z = \sqcup Z_i$ where each component of $\Gamma_{Z_i}$ has order $i$. Let $r$ be the first index such that $Z_r \ne \varnothing$. The new components generated by $\Gamma/\Gamma_{Z_r}$ keep order $r$. The order of the new components in $\hat{\Gamma}_{Z^{sst}_r}$ is $r+1$. Now we increase by one the order of all unaffected components except for those of order less than or equal to $r$. This defines an order function on $\Gamma/Z_r$. By remark~\ref{splitcollapse} we can continue this process inductively until we generate an order function for $\Gamma/Z$.

\item There are possibly new induced nodes together with an involution without fixed points. In the case of the total collapse of a homotopy circle without labeled points that corresponds to a cusp-node the old vertex-node and the newly generated vertex-node are associated to both cusp-nodes of the generated semistable circle. It can be showed following the inductive construction in the previous item that the resulting involution without fixed point satisfies the definition required by a semistable ribbon graph.

\end{itemize} \label{semiedgecollapse}
\end{df}

The previous definition is really a lemma which we state below.

\begin{lem} \label{semi}
The edge collapse of a collapsible subset of a $P$-labeled semistable ribbon graph produces a new $P$-labeled semistable ribbon graph of the same topological type but with possibly more components of higher order and more nodes.
\end{lem}

\begin{rem}
To obtain semistable ribbon graphs with higher orders we need to collapse several subsets of a ribbon graph consecutively. Therefore, the right notion of edge collapse in a category of $P$-labeled semistable ribbon graphs is that of consecutive collapse of collapsible subsets.
\end{rem}

\subsection{Permissible Sequences}

Fix a pair of associated nodes on a $P$-labeled semi-stable ribbon graph. A {\bf tangent direction} is a choice of gluing between the vertex-node and the boundary cycle corresponding to the cusp-node as in Definition~\ref{gluing}. This choice has to be compatible with the cyclic orders on the set of half-edges of the vertex-node and the edges of the graph associated to the exceptional boundary cycle corresponding to the cusp-node. We are just choosing then an element of the finite set of isomorphism classes of graphs created by the gluing construction.

\begin{df}
A {\bf decoration by tangent directions} on a semistable ribbon graph is the choice of tangent directions for each pair of associated nodes.
\end{df}

An {\bf isomorphism} of semistable ribbon graphs decorated by tangent directions must preserve the tangent directions in the sense that the there is an induced graph isomorphism on the corresponding gluings.

The previous definition of semistable ribbon graphs decorated by tangent directions will connect graphs with complex surfaces after introducing metrics on ribbon graphs. The following approach is better suited to induce a topology in the combinatorial moduli space that we will later define.

\begin{df}
Given a $P$-labeled ribbon graph $\Gamma$, a {\bf permissible sequence} is a sequence \[ Z_\bullet = (E(\Gamma)=Z_0,Z_1,...,Z_k) \] such that $Z_i \subset Z_{i-1}^{sst}$ where the inclusion is strict. We call $k$ the length of the sequence. The pair $(\Gamma, Z_\bullet)$ denotes a labeled ribbon graph and a permissible sequence in it. If in addition all $Z_i$'s are semistable we call this a {\bf semistable sequence}. An {\bf isomorphism} of ribbon graphs with permissible sequences is a ribbon graph isomorphism that preserve the permissible sequences. 
\end{df}

\begin{rem}
The length of the sequence will correspond with the maximal order of an associated semistable ribbon graph. Notice also that there is a natural bijection between pairs of length zero and $P$-labeled ribbon graphs.
\end{rem}

\begin{df}
A {\bf negligible} subset of $(\Gamma,Z_\bullet)$ is a sequence $D_\bullet = (D_0, D_1,...,D_k)$ such that all $D_i$ are negligible, $D_i \subset D_{i-1}$ and $D_i \subset Z_i$. Call ${\mathcal N}(\Gamma,Z_\bullet)$ the set of negligible subsets of $(\Gamma,Z_\bullet)$. 
\end{df}

\begin{rem} \label{negli}
It is easy to check that we have a bijection between negligible subsets of $\Gamma$ and negligible subsets of $(\Gamma,Z_\bullet)$ by using the natural restriction. Moreover, we can collapse along negligible subsets in a similar way as we did before. Given a permissible sequence $Z_\bullet$ and negligible subset $D_\bullet$ we define the {\bf edge collapse} of $(\Gamma,Z_\bullet)$ along $D_\bullet$ as $(\Gamma / \Gamma_{D_0}, (Z/ D)_\bullet)$ where $(Z / D)_\bullet$ is the sequence induced by edge collapse. It can be shown that the result is also permissible and has the same length. 
\end{rem}

Now that we know how to collapse along negligible subsets, we also want to be able to collapse permissible sequences along semistable subsets but we need to be careful on how we define the new sequence. Let $(\Gamma,Z_\bullet)$ be a $P$-labeled ribbon graph together with a permissible sequence. A subset $S \subset E(\Gamma)$ is {\bf collapsible} with respect to $(\Gamma, Z_\bullet)$ if $Z_i \not\subset S$ for all $i$. This last definition is similar to the concept of collapsible subset for semistable ribbon graphs and serves the same function.

\begin{lem} \label{semistable}
Given a collapsible subset $S$ with respect to $(\Gamma,Z_\bullet)$ and semistable in $\Gamma$, we can induce a new permissible sequence $(Z / S)_\bullet$ inductively. 
\end{lem}

\begin{proof}
Let $i$ be the integer satisfying $S \subset Z_i$ and $S \subset \!\!\!\!\!\! / \,\,\, Z_{i+1}$. Then $(Z/S)_j = Z_j$ for $j \le i$. Set $(Z/S)_{i+1} = S \cup Z_{i+1}$ and $(Z/S)_{i+2} = Z_{i+1}$. Now, if $S\cap (Z_{i+1} - Z_{i+2}) \ne \varnothing$ then $(Z/S)_{i+3} = (S - Z_{i+1}^c ) \cup Z_{i+2}$ and $(Z/S)_{i+4} = Z_{i+2}$, otherwise $(Z/S)_{i+3} = Z_{i+2}$. We can continue this process until the we reach the last step: either we exhaust all of $S$ meaning that the last element of the sequence will be $(Z/S)_l = Z_k$ or $(Z/S)_l = S - Z_k^c$ where $k$ is the length of $Z_\bullet$ and $l$ the length of the new sequence. The resulting sequence can be shown to be permissible and will have $l>k$. The resulting pair is then $(\Gamma, (Z/S)_\bullet)$. 
\end{proof}

\begin{prop}  \label{corres}
A $P$-labeled ribbon graph together with a permissible sequence $Z_\bullet$ can be used to construct a $P$-labeled semistable ribbon graph.
\end{prop}

\begin{proof}
For $i>0$ we can always collapse $Z_i - Z_i^{sst}$ since these sets are negligible due to maximality. Therefore we can assume that all $Z_i$ are semistable for $i>0$. The disjoint union $\Gamma / \Gamma_{Z_1} \sqcup \hat{\Gamma}_{Z_1}$ naturally inherits a semistable ribbon graph structure through the involution identifying exceptional vertices with their corresponding exceptional boundary cycles. The connected components of $\hat{\Gamma}_{Z_1 - Z_1^{st}}$ are semistable circles. The components in $\Gamma / \Gamma_{Z_1}$ only contain vertex-nodes and thus all those components have order zero. All the components of $\hat{\Gamma}_{Z_1}$ have at least one cusp-node associated to a vertex-node in a component of order zero and hence all those components have order one. The $P$-labeling naturally induces a $P$-labeling on the semistable ribbon graph. We can inductively apply this process to $\hat{\Gamma}_{Z_i}$ and $Z_{i+1}$ thus obtaining a $P$-labeled semistable ribbon graph $(\Gamma / \Gamma_{Z_1} \sqcup \hat{\Gamma}_{Z_1} / \Gamma_{Z_2} \sqcup \cdots \sqcup \hat{\Gamma}_{Z_k},\iota,x)$. 
\end{proof}

Now we describe the connection between ribbon graphs with semistable sequences and semistable ribbon graphs with decorations by tangent directions. 

\begin{thm} \label{bijec}
There is a natural bijection between isomorphism classes of $P$-labeled ribbon graphs with semistable sequences and isomorphism classes of $P$-labeled semistable ribbon graphs with decorations by tangent directions. This identification preserves isomorphism classes of negligible and collapsible semistable subsets (with respect to the given structures) and commutes with the edge collapse of the corresponding sets.
\end{thm}

\begin{proof}
Let $\Gamma$ be a $P$-labeled ribbon graph and $Z_\bullet$ a semistable sequence. This generates a $P$-labeled semistable ribbon graph by Proposition~\ref{corres}. To obtain the decorations by tangent directions, it is enough to keep track of where the half-edges of a vertex-node were attached on the original graph. This correspondence naturally descends to a correspondence on isomorphism classes.

Now suppose we have a $P$-labeled semistable ribbon graph decorated by tangent directions. The decorations by tangent directions allow us to reconstructs a $P$-labeled ribbon graph by using the gluing construction on vertex-nodes and boundary cycles. Since this is defined only up to isomorphism this correspondence is well defined on isomorphism classes. On a representative, every component of a semistable graph induces a subgraph of the ribbon graph. Together with the order this defines a sequence of subgraphs $Z_\bullet$ in the ribbon graph up to isomorphism. It is not hard to check that this sequence will indeed be semistable.

These correspondences are inverses of each other on isomorphism classes by construction. Remark~\ref{negli} implies that negligible subsets are preserved and it also implies the commutativity with the edge collapse. For collapsible semistable subsets we also use the natural restriction and the gluing construction to track the image of these sets under the bijection. By the definitions, Lemma~\ref{semi} and Lemma~\ref{semistable} we can show that collapsible semistable subsets are also preserved by the bijection.
\end{proof}

\begin{rem}
In fact it is possible to define a category of semistable ribbon graphs and another one of ribbon graphs with permissible sequences. After defining the right notion of morphism the previous theorem can be extended to an equivalence of appropriate categories.
\end{rem}

\section{Cellular Decompositions} \label{celudecom}

\subsection{Metrics on Ribbon Graphs}

\begin{df}
A {\bf metric} on a ribbon graph $\Gamma$ is a map $l:E(\Gamma) \to \R_+$. If the sum of the lengths of all edges is one we call this a {\bf unital metric} or {\bf conformal structure}. A unital metric on a semistable ribbon graph is a sequence $\{l_\bullet\}$ of unital metrics on every union of connected components of a fixed order. We call such structure a {\bf conformal semistable metric}.
\end{df}

Notice that the surface $\surf(\Gamma) - C(\Gamma)$ inherits a piece-wise Euclidean metric induced by the lengths of the edges. 

An {\bf isomorphism of metric ribbon graphs} is a ribbon graph isomorphism that respects the metric. The space of conformal structures on $\Gamma$ up to isomorphism will be denoted by $cf(\Gamma)$. We use the same notation when $\Gamma$ and the metric are semistable. If the ribbon graphs are $P$-labeled we require such isomorphism to fix the labels pointwise. A point in $cf(\Gamma)$ can be denoted by $\Gamma_\text{met}$. The main consequence of having a metric on a ribbon graph is the following.

\begin{prop} \label{metric}
A metric on a ribbon graph induces a complex structure on the surface it determines.
\end{prop}

\begin{proof} This is the reason why $\surf(\Gamma)$ was constructed out of patches of the complex plane.

Now that every edge has a well-defined length, the tiles $K_e$ are subsets of the complex plane. It is then possible to give $\surf(\Gamma) - \{ \Gamma \cup P \}$ a canonical atlas of complex charts. Such complex structure extends to $\surf(\Gamma)$ making this a compact Riemann surface with $P$-labeled points denoted by $C(\Gamma,l)$ (see \cite[Theorem 5.1]{mupe} and \cite[6.2]{loo}).

\end{proof}

The previous construction can be carried out on the irreducible components of a semistable surface. Therefore, given a conformal semistable ribbon graph we can induce a conformal structure on the singular surface it determines.

\begin{rem} \label{orbicell}
There is a natural identification \[ cf(\Gamma) = \left( \prod_{k \ge 0} \stackrel{\circ}{\Delta}_{E(\Gamma_k)}\right) / G \] where $\Gamma_k$ is the subgraph containing all components of order $k$, $\stackrel{\circ}{\Delta}_{E(\Gamma_k)}$ is the open simplex generated by the set of edges of $\Gamma_k$ and $G$ is a finite group acting by automorphisms of metric ribbon graphs. This is thus a rational cell following the language of \cite{mupe} which we call an {\bf orbicell}. 
\end{rem}

Now we follow the notation in sections 2 and 3 of \cite{mupe}. We use their definition of orbifold, differentiable orbifold and orbifold-cell decomposition which we are calling an {\bf orbicell decomposition} of an orbifold. For an alternate definition one can check the Appendix in \cite{cos:dpv}.

\begin{df} A {\bf near conformal structure} on a ribbon graph $\Gamma$ is a conformal structure $l:E(\Gamma) \to \R_{\ge 0}$ whose zero set is negligible. The space of near conformal structures is denoted by $ncf(\Gamma)$. 
\end{df}

\begin{df}
Given a $P$-labeled ribbon graph and a permissible sequence $Z_\bullet$ a {\bf semistable conformal structure} with respect to such a sequence is a conformal structure on every difference $\Gamma_{Z_k - Z_{k+1}}$.
\end{df}

\begin{rem} \label{stm} 
From the previous definition we can see that a semistable conformal metric may be given as a sequence of functions $l_k:Z_k \to \R_{\ge 0}$ such that $l_k$ has zero set $Z_{k+1}$ (so $l_\bullet$ determines $Z_\bullet$) and the total length of each $Z_k$ adds up to one. We can thus define the spaces $cf(\Gamma,Z_\bullet)$ and $ncf(\Gamma, Z_\bullet)$.
\end{rem}

Now we construct an orbicell decomposition made out of semistable ribbon graphs.

\begin{df}
The {\bf moduli space of $P$-labeled semistable ribbon graphs of genus $g$ decorated by tangent directions} is defined as \[ \MUC{g}{P} = \coprod_{[(\Gamma,Z_\bullet)]} cf(\Gamma, Z_\bullet) \] where the union is taken over isomorphism classes of $P$-labeled semistable ribbon graphs with decorations by tangent directions of topological type $(g,|P|)$ and permissible sequences.
\end{df}

\begin{thm}
The set $\MUC{g}{P}$ has a natural structure of a topological space.
\end{thm}

\begin{proof}
The topology of the orbicell decomposition is determined by how the orbicells are glued together. Two orbicells are glued when one can be obtained from the other by collapsing edges. Given any non-empty proper subset of edges $Z_1$ we can glue a new orbicell along the boundary (notice that the properness is necessary since the sum of edges always adds up to one). If $Z_1$ is negligible this is just part of $ncf(\Gamma)$ which gives a partial compactification. Otherwise take $Z_1 - Z_1^{sst}$ first, and then glue along $cf(\Gamma,(E(\Gamma),Z_1 - Z_1^{sst}))$. However there might be missing pieces of the boundary. Those pieces correspond to possible degenerations of $\Gamma_{Z_1^{sst}}$. This process can be understood as using $ncf(\Gamma)$ to glue orbicells. If we continue this way we can inductively glue orbicells corresponding with semistable ribbon graphs of higher order.

Now we describe a system of neighborhoods that generate the topology of $\MUC{g}{P}$. Recall from \cite[Section 3]{mupe} that we write $\Gamma_1 \prec \Gamma_2$ when $\Gamma_1$ can be obtained from $\Gamma_2$ by edge collapse. We also say then that $\Gamma_2$ is obtained by {\bf edge expansion} of $\Gamma_1$. This definition can be extended to $P$-labeled semistable ribbon graphs in a natural way due to Definition~\ref{semiedgecollapse}. This implies that the edge expansion also includes desingularization of graphs. Given a $\Gamma_\text{met} \in \MUC{g}{P}$ let $\epsilon>0$ be a positive number smaller than half of the length of the shortest edge of $\Gamma_\text{met}$. The {\bf $\epsilon$-neighborhood} of $\Gamma_\text{met}$ in $\MUC{g}{P}$, denoted by $U_\epsilon (\Gamma_\text{met})$, is the set of all $P$-labeled semistable metric ribbon graphs $\Gamma_\text{met}'$ satisfying the following conditions.
\begin{itemize}
\item $\Gamma \preceq \Gamma'$.
\item The edges of $\Gamma_\text{met}'$ that are contracted into $\Gamma_\text{met}$ have length less than $\epsilon$.
\item Let $e'$ be an edge of $\Gamma'_\text{met}$ that is not contracted and corresponds to an edge $e$ of $\Gamma_\text{met}$ of length $L$. Then, the length $L'$ of $e'$ is in the range \[ L - \epsilon < L' < L+\epsilon. \]
\item The lengths of the edges in $\Gamma'_\text{met}$ are chosen so that the metric is still a conformal semistable metric.
\end{itemize}
For non-singular graphs and possibly non-unital metrics this is the same as \cite[Definition 3.1]{mupe}. The topology of $\MUC{g}{P}$ is defined as the smallest topology that has these $\epsilon$-neighborhoods as open sets.
\end{proof}

\begin{rem}
In fact it is possible to extend the proof of \cite[Theorem 3.5]{mupe} to our case in order to show that $\MUC{g}{P}$ is a differentiable orbifold.
\end{rem}

By forgetting the decorations by tangent directions we obtain the following definition.

\begin{df}
The {\bf moduli space of $P$-labeled semistable ribbon graphs of genus $g$} is defined as \[ \MOC{g}{P} = \coprod_{[\Gamma]} cf(\Gamma) \] where $[\Gamma]$ is an isomorphism class of $P$-labeled semistable ribbon graph of topological type $(g,|P|)$.
\end{df}

In light of Remark~\ref{stm} and Theorem~\ref{bijec}, these orbicell decompositions can be defined in terms of $P$-labeled ribbon graphs together with conformal semistable metrics. This comes with a map $\MUC{g}{P} \to \MOC{g}{P}$ induced by the map forgetting the decorations by tangent directions on a semistable ribbon graph. The preimage of a point is the space of decorations by tangent directions on a particular class of conformal semistable ribbon graph. This map allow us to induce the quotient topology on $\MOC{g}{P}$ using the previous theorem.

\begin{ex}
To visualize some orbicells and how they fit together consider the space $\MOC{0}{P}$ where $|P|=4$. Figure~\ref{maximalcell} shows a trivalent graph and how it can degenerate to two different semistable ribbon graphs. The number over a component of a graph denotes its order. The trivalent graph determines an orbicell of the form $\stackrel{\circ}{\Delta}_5$ whose dimension agrees with the dimension of the corresponding moduli space. If we collapse the subgraph determined by the big circle and what it is inside it we obtain the graph on the bottom left. If we collapse only the big circle we obtain the graph on the bottom right. The singular graph on the bottom left corresponds with an orbicell of the form $\stackrel{\circ}{\Delta}_1 \times \stackrel{\circ}{\Delta}_2$ and the one on the bottom right with an orbicell of the form $\stackrel{\circ}{\Delta} _3 \times \stackrel{\circ}{\Delta}_0$ where $\stackrel{\circ}{\Delta}_0$ comes from the semistable circle. This last graph has five edges, but only four of them can be collapsed since the semistable circle can not any more. Those four edges are labeled $a$, $b$, $c$ and $d$. Since this last orbicell is three-dimensional we show in Figure~\ref{deginci} this orbicell together with its degenerations. The straight arrows correspond with faces on the front and the curved arrows with faces on the back.
\end{ex}

\begin{figure} 
\begin{tikzpicture}[scale=0.35]
% First graph
\draw (0,0) circle (1cm);
\draw (6,0) circle (1cm);
\draw (6,0) circle (3cm);
\draw (1,0) -- (3,0);
\draw (7,0) -- (9,0);

\draw[black, fill=black] (1,0) circle (5pt);
\draw[black, fill=black] (3,0) circle (5pt);
\draw[black, fill=black] (7,0) circle (5pt);
\draw[black, fill=black] (9,0) circle (5pt);
\node at (4,5) {0};

\draw[thick,->] (0,-4) -- (-3,-6);
\draw[thick,->] (8,-4) -- (11,-6);

%Left graph

\draw (-11,-12) circle (1cm);
\draw (-10,-12) -- (-8,-12);
\draw (-3,-12) circle (1cm);
\draw (-3,-12) circle (3cm);
\draw (-2,-12) -- (0,-12);

\draw[black, fill=black] (-10,-12) circle (5pt);
\draw[black, fill=black] (-8,-12) circle (5pt);
\draw[black, fill=black] (-2,-12) circle (5pt);
\draw[black, fill=black] (0,-12) circle (5pt);

\node at (-10,-8) {0};
\node at (-3,-8) {1};

%Right graph

\draw (8,-12) circle (1cm);
\draw (9,-12) -- (11,-12);
\draw (14,-12) circle (1cm);
\draw (17,-12) -- (19,-12);
\draw (20,-12) circle (1cm);

\draw[black, fill=black] (9,-12) circle (5pt);
\draw[black, fill=black] (11,-12) circle (5pt);
\draw[black, fill=black] (17,-12) circle (5pt);
\draw[black, fill=black] (19,-12) circle (5pt);

\node at (9,-9) {0};
\node at (14,-9) {1};
\node at (19,-9) {0};

\node at (8,-14) {$a$};
\node at (10,-13.5) {$b$};
\node at (18,-13.5) {$c$};
\node at (20,-14) {$d$};

\end{tikzpicture}
\caption[Degeneration of trivalent graph]{Two degenerations of a trivalent graph in $\MOC{0}{4}$.} \label{maximalcell}
\end{figure}
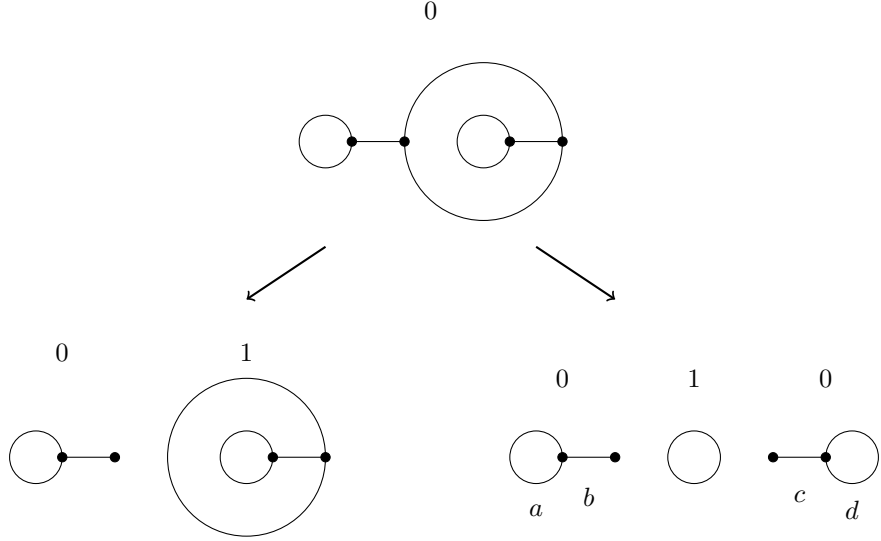

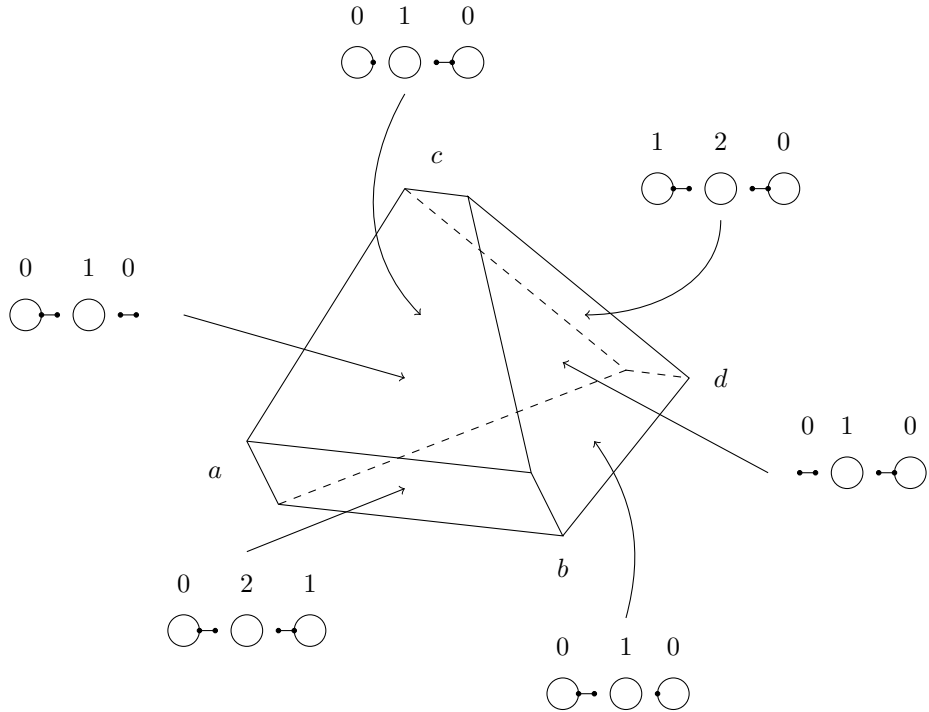
\begin{figure} 
\begin{tikzpicture}[scale=0.42]

\draw[dashed] (4,10) -- (11,4.25) -- (0,0);
\draw[dashed] (11,4.25) -- (13,4);

\draw (9,-1) -- (13,4) -- (6,9.75) -- (8,1) -- (9,-1) -- (0,0) -- (-1,2) -- (4,10) -- (6,9.75);
\draw (-1,2) -- (8,1);

\node at (-2,1) {$a$};
\node at (9,-2) {$b$};
\node at (5,11) {$c$};
\node at (14,4) {$d$};

% 1st Graph
\draw (-3,-4) circle (0.5cm);
\draw (-2.5,-4) -- (-2,-4);
\draw (-1,-4) circle (0.5cm);
\draw (0,-4) -- (0.5,-4);
\draw (1,-4) circle (0.5cm);
\draw [black, fill=black] (-2.5,-4) circle (2pt);
\draw [black, fill=black] (-2,-4) circle (2pt);
\draw [black, fill=black] (0,-4) circle (2pt);
\draw [black, fill=black] (0.5,-4) circle (2pt);
\node at (-3,-2.5) {0};
\node at (-1,-2.5) {2};
\node at (1,-2.5) {1};

\draw[->] (-1,-1.5) -- (4,0.5);

% 2nd Graph
\draw (-3+15,-4+14) circle (0.5cm);
\draw (-2.5+15,-4+14) -- (-2+15,-4+14);
\draw (-1+15,-4+14) circle (0.5cm);
\draw (0+15,-4+14) -- (0.5+15,-4+14);
\draw (1+15,-4+14) circle (0.5cm);
\draw [black, fill=black] (-2.5+15,-4+14) circle (2pt);
\draw [black, fill=black] (-2+15,-4+14) circle (2pt);
\draw [black, fill=black] (0+15,-4+14) circle (2pt);
\draw [black, fill=black] (0.5+15,-4+14) circle (2pt);
\node at (-3+15,-2.5+14) {1};
\node at (-1+15,-2.5+14) {2};
\node at (1+15,-2.5+14) {0};

\draw[->] (-1+15,-2.5+11.5) to [out=270,in=0] (9.7,6);

% 3rd Graph
\draw (-3-5,-4+10) circle (0.5cm);
\draw (-2.5-5,-4+10) -- (-2-5,-4+10);
\draw (-1-5,-4+10) circle (0.5cm);
\draw (0-5,-4+10) -- (0.5-5,-4+10);
%\draw (1-5,-4+10) circle (0.5cm);
\draw [black, fill=black] (-2.5-5,-4+10) circle (2pt);
\draw [black, fill=black] (-2-5,-4+10) circle (2pt);
\draw [black, fill=black] (0-5,-4+10) circle (2pt);
\draw [black, fill=black] (0.5-5,-4+10) circle (2pt);
\node at (-3-5,-2.5+10) {0};
\node at (-1-5,-2.5+10) {1};
\node at (0.25-5,-2.5+10) {0};

\draw[->] (-3,6) -- (4,4);

% 4th Graph
\draw (-3+12,-6) circle (0.5cm);
\draw (-2.5+12,-6) -- (-2+12,-6);
\draw (-1+12,-6) circle (0.5cm);
%\draw (0+12,-6) -- (0.5+12,-6);
\draw (0.5+12,-6) circle (0.5cm);
\draw [black, fill=black] (-2.5+12,-6) circle (2pt);
\draw [black, fill=black] (-2+12,-6) circle (2pt);
\draw [black, fill=black] (0+12,-6) circle (2pt);
%\draw [black, fill=black] (0.5+12,-6) circle (2pt);
\node at (-3+12,-4.5) {0};
\node at (-1+12,-4.5) {1};
\node at (0.5+12,-4.5) {0};

\draw [->] (11,-3.6) to [out=75,in=-55] (10,2);

% 5th Graph
%\draw (-3+19,1) circle (0.5cm);
\draw (-2.5+19,1) -- (-2+19,1);
\draw (-1+19,1) circle (0.5cm);
\draw (0+19,1) -- (0.5+19,1);
\draw (1+19,1) circle (0.5cm);
\draw [black, fill=black] (-2.5+19,1) circle (2pt);
\draw [black, fill=black] (-2+19,1) circle (2pt);
\draw [black, fill=black] (0+19,1) circle (2pt);
\draw [black, fill=black] (0.5+19,1) circle (2pt);
\node at (-2.5+19.25,2.5) {0};
\node at (-1+19,2.5) {1};
\node at (1+19,2.5) {0};

\draw [->] (15.5,1) -- (9,4.5);

% 6th Graph
\draw (-3+5.5,-4+18) circle (0.5cm);
%\draw (-2.5+5,-4+18) -- (-2+5,-4+18);
\draw (-1+5,-4+18) circle (0.5cm);
\draw (0+5,-4+18) -- (0.5+5,-4+18);
\draw (1+5,-4+18) circle (0.5cm);
%\draw [black, fill=black] (-2.5+5,-4+18) circle (2pt);
\draw [black, fill=black] (-2+5,-4+18) circle (2pt);
\draw [black, fill=black] (0+5,-4+18) circle (2pt);
\draw [black, fill=black] (0.5+5,-4+18) circle (2pt);
\node at (-3+5.5,-2.5+18) {0};
\node at (-1+5,-2.5+18) {1};
\node at (1+5,-2.5+18) {0};

\draw [->] (4,13) to [out=-120,in=135] (4.5,6);

\end{tikzpicture}
\caption[Incidence of orbicells]{Degenerations of the second singular graph in Figure~\ref{maximalcell} and how the corresponding orbicells fit together.} \label{deginci}
\end{figure}

\subsection{Strebel-Jenkins Differentials} A {\bf meromorphic quadratic differential} on a Riemann surface $C$ is a meromorphic section of $(T^*C)^{\odot 2}$, the second symmetric power of the cotangent bundle. The notions of zero and order of a zero of these differentials do not depend on the local representation. In the same way the notion of pole and order of a pole are stable by change of coordinates. Zeros and poles will be call {\bf critical points}. If the quadratic differential has a pole of order two this is called a {\bf double pole} and a pole of order one a {\bf simple pole}. Given a representation in local coordinates $f(z)dz^2$ around a double pole $q$ we can express $f$ as

$$ f(z) = \frac{a_{-2}}{z^2} + \frac{a_{-1}}{z} + a_0 + \cdots $$

\noindent and call the term $a_{-2}$ its {\bf quadratic residue}. It can be shown that this number does not depend on the choice of local coordinates. 

These differentials define certain curves on the Riemann surface. If $q=f(z)dz^2$ is a meromorphic quadratic differential then the parametric curve $\vec{r}:(a,b) \to C$ is called a {\bf horizontal trajectory} or {leaf} of $q$ if \[ f(\vec{r}(t)) \left( \frac{d\vec{r}(t)}{dt} \right)^2 > 0 \] and {\bf vertical trajectory} if \[ f(\vec{r}(t)) \left( \frac{d\vec{r}(t)}{dt} \right)^2 < 0. \] 

The quadratic differentials we are particularly interested on are the following.

\begin{df}
A {\bf Strebel-Jenkins differential} is a meromorphic quadratic differential with only simple poles or double poles with negative quadratic residues.
\end{df}

In the case of Strebel-Jenkins differentials we have two kinds of leaves: closed ones (surrounding a double pole) and critical ones (connecting zeroes and simple poles). The union of critical leaves, zeroes and simple poles forms the {\bf critical graph}. The vertical trajectories connect the double poles to the critical graph and are orthogonal to the closed leaves under the metric induced by $\sqrt{q}$. The following existence and uniqueness theorem follows from the work of Jenkins and Strebel (see \cite{strebel} and \cite[Theorem 7.6]{loo}).

\begin{figure}
$$
\begin{array}{ccc}
\includegraphics[width=2.9cm]{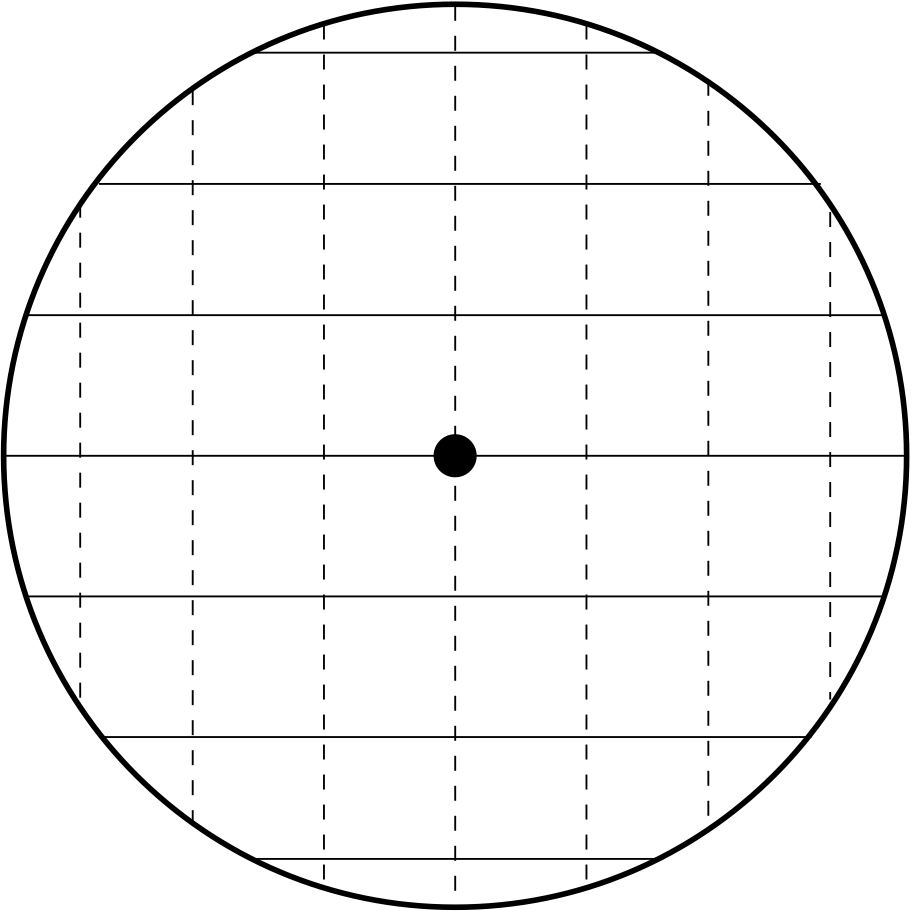} & \quad \includegraphics[width=3cm]{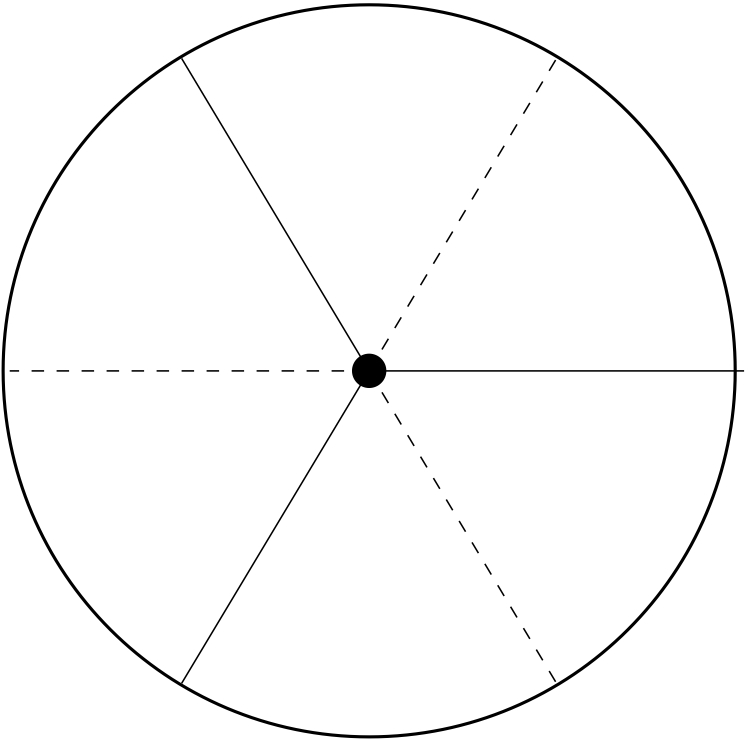} & \quad \includegraphics[width=3cm]{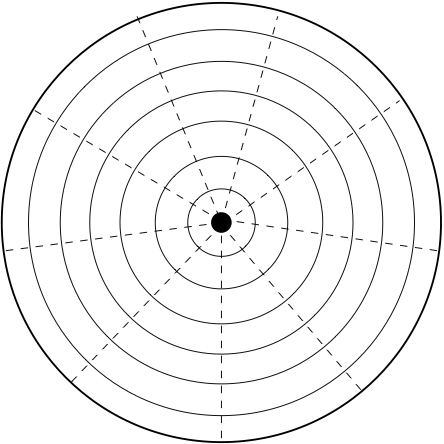} \\
q = dz^2 & \quad q = z^m dz^2 & \quad q = - \frac{dz^2}{z^2}
\end{array}
$$
\caption{Different behaviors of Strebel-Jenkins differentials. The solid lines represent horizontal trajectories and the dotted one vertical trajectories.}
\end{figure}

\begin{thm} \label{strebel}
Given a Riemann Surface of genus $g$ with labeled points $P$ and decorations $\lambda \in \Delta_P$ there exists a unique quadratic differential with the following properties. It is holomorphic on the complement of $P$. The union of closed leaves form semi-infinite cylinders around the points with non-zero decoration. The quadratic residues coincide with $\lambda$. The labeled points decorated by zero lie on the critical graph. 
\end{thm}

If we restrict to connected (not necessarily unital) metric ribbon graphs with vertices of valence at least three and then put together orbicells as in Remark~\ref{orbicell}, this gives the space $\MCgP$ of $P$-labeled ribbon graphs as in \cite{mupe}. The map $\Psi : \MCgP \to \MDgP$ uses the construction of Proposition~\ref{metric}. The decorations come from taking half the perimeter of the subgraph associated to a boundary cycle or it is zero if the labeled point lies on the graph. The reason why we take the half is because each edge is counted twice, one for each orientation. Theorem~\ref{strebel} provides its inverse. As these maps are continuous $\Psi$ is a homeomorphism.

Now we describe an extension of $\Psi$. The map $\ul{\Psi}: \MUC{g}{P} \to \MUD{g}{P}$ is well defined for non-singular graphs and surfaces. Let $[\Gamma]$ be a point in $cf (\Gamma, Z_\bullet)$. The metric clearly defines a semistable Riemann surface with the aid of the involution $\iota$ and the order function. There is also an induced $P$-labeling. The decoration at each labeled point is induced by taking half the length of the corresponding subgraph associated to a boundary cycle or it is zero when the labeled point lies on the graph. 

To induce decorations by tangent directions we follow the idea illustrated in Figure~\ref{tandir}. On the blowup of a vertex-node choose a parametrization making one of the half-edges coincide with the positive real line and so that there is an equal distance between each half-edge. The reason for choosing this particular parametrization is to make this construction compatible with the complex chart induced at a vertex of a metric ribbon graph (see \cite[Theorem 5.1]{mupe}). The half-edge on the positive real line induces a tangent vector $z_1$ on the induced surface. On the cusp-node there is a natural parametrization of the boundary subgraph by $\S^1$ with opposite orientation up to rotation. This is because the graph has a metric and thus the subgraph associated to the boundary cycle has a well defined length that can be rescaled. Since the graph has a decoration by tangent directions the half-edge on the vertex-node corresponding to the positive real line induces a point on the subgraph associated to the boundary cycle. To fix the parametrization of the boundary subgraph let the positive real line coincide with the previously induced point. This point in turn induces a tangent vector $z_2$ on the node of the surface by going along a vertical trajectory starting at the induced point and taking minus the tangent of such trajectory at the node. Now let the decoration by tangent directions on the surface be $z_1 \otimes z_2$. It is not hard to show that this definition is independent of the choices up to surface isomorphism.

\begin{figure} 
\includegraphics[width=12.5cm]{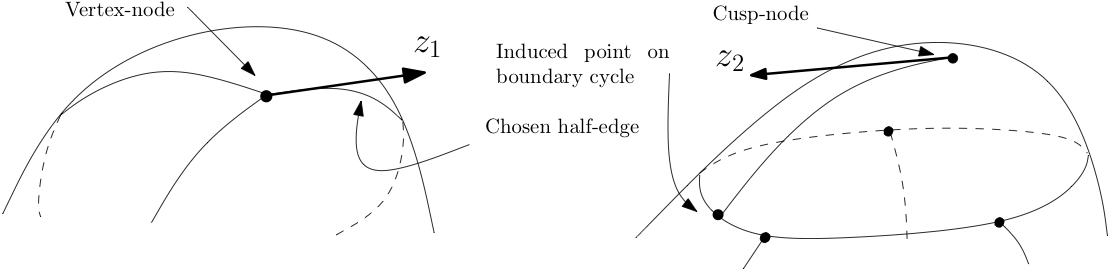}
\caption[Tangent Directions.]{Tangent directions on a semistable ribbon graph inducing tangent directions on the corresponding Riemann surface.} \label{tandir}
\end{figure}

The second main theorem of this paper is the following.

\begin{thm} \label{hmuc}
The map $\ul{\Psi}: \MUC{g}{P} \to \MUD{g}{P}$ is a homeomorphism.
\end{thm}

The proof is a generalization of \cite[Theorem 11.5]{loo} quoted below. This generalization requires a careful analysis of the decorations on labeled points and decorations by tangent directions which we now present.

Using again Strebel-Jenkins differentials we can produce the inverse $\ul{\Psi}^{-1}$ making this map a bijection. This inverse assigns to a class of a $P$-labeled decorated semistable Riemann surfaces the isomorphism class of a metric semistable ribbon graph via Theorem~\ref{strebel}.

By forgetting the decorations by tangent directions we get another surjection $\ol{\Psi} : \MOC{g}{P} \to \MOD{g}{P}$. Finally by forgetting the decorations and order on the semistable Riemann surfaces we get surjections $\ol{\Phi}: \MOC{g}{P} \to \MO{g}{P}$ and $\ul{\Phi}: \MUC{g}{P} \to \MU{g}{P}$.

Using the notation of \cite{loo} we have a projection $\MOC{g}{P} \to \Gamma \backslash \hat{A}$  where $\hat{A}$ is a cellular decomposition of the Teichm\"{u}ller analogue for the compactified decorated moduli space related to the arc complex and $\Gamma$ is the mapping class group acting on $\hat{A}$. The definition of $\hat{A}$ is connected to the definition of the combinatorial moduli space by taking the dual graph. The preimage of a point under this map corresponds with decorations including semistable spheres. In fact we have the following

\begin{thm} (Looijenga) The map $\Gamma \backslash \hat{A} \to \MO{g}{P}$ is a continuous surjection with the preimage of a point being the space of decorations by non-negative real numbers after collapsing semistable spheres.
\end{thm}

By keeping track of the extra decorations on semistable spheres and using the projection $\MOC{g}{P} \to \Gamma \backslash \hat{A}$ we can extend Looijenga's main theorem to the following result.

\begin{thm} The map $\ol{\Phi}: \MOC{g}{P} \to \MO{g}{P}$ is a continuous surjection with  the preimage of a point being the space of all semistable ribbon graphs generating the same conformal class in the Deligne-Mumford moduli space.
\end{thm}

The following result easily follows from the previous one by keeping track of the decoration by tangent directions. It never appeared in the literature because the space $\MUC{g}{P}$ is new.

\begin{prop}
The map $\ul{\Phi}: \MUC{g}{P} \to \MU{g}{P}$ is a continuous surjection with preimages the space of all semistable ribbon graphs decorated by tangent directions generating the same conformal class in the real oriented blowup of the Deligne-Mumford moduli space.
\end{prop}

In order to extend $\Psi$ to the boundary we need to extract decorations from a metric on a ribbon graph. Given a metric ribbon graph $(\Gamma,l)$ we can construct a function $\lambda: C(\Gamma) \to \R_+$ defined as half the total length of the associated boundary subgraph (counting twice those edges with both half-edges in the boundary cycle). This is called a {\bf perimeter function}. For a metric $P$-labeled semistable ribbon graph the perimeter function is defined by $\lambda: x(P) \sqcup N \to \R_{\ge 0}$ vanishing only at the points that correspond with vertices of the graph  and assigning to each cusp half the perimeter of the corresponding boundary subgraph (counting twice those edges with both half-edges in the boundary cycle).

It is now possible to redefine the maps $\ol{\Psi}=(\ol{\Phi},\lambda)$ and $\ul{\Psi}=(\ul{\Phi},\lambda)$.

\begin{thm}
The map $\ol{\Psi}: \MOC{g}{P} \to \MOD{g}{P}$ is a homeomorphism. \label{mainthm3}
\end{thm}

\begin{proof}
The function $\ol{\Psi}$ has an inverse constructed from Strebel's theorem. Such inverse assigns to a decorated $P$-labeled semistable Riemann surface a $P$-labeled semistable ribbon graph with the metric induced from the conformal structure on the surface and transferring the order function to the graph component by component. This makes $\ol{\Psi}$ a bijection. Since both spaces are Hausdorff and compact it is enough to show continuity of $\ol{\Psi}$ to show that it is a homeomorphism. The continuity of $\ol{\Phi}$ can be extended to the continuity of $\ol{\Psi}$ by keeping track of the decorations by non-negative real numbers.
\end{proof}

\begin{proof}[Proof of Theorem~\ref{hmuc}] This is a generalization of the previous theorem obtained by keeping track of the decorations by tangent directions.
\end{proof}

\begin{rem}
The continuity of $(\ol{\Psi})^{-1}$ can be proved provided one can extend the proof in \cite{zvon} by a careful analysis of the convergence of Strebel-Jenkins differentials via the normalization described in Section~\ref{compactifications}.
\end{rem}

One of the difficulties in showing the existence of the homeomorphisms $\ol{\Psi}$ and $\ul{\Psi}$ arises from defining the spaces $\MUD{g}{P}$ and $\MOD{g}{P}$ precisely. This allows us to interpret the space of decorations as combinatorial data that is possible to embed in the definition of a stable Riemann surface thus giving the desired homeomorphisms.

\begin{coro}
We also get orbicell decompositions of $\MOC{g}{P} / \FS_P$, $\MUC{g}{P} / \FS_P$ homeomorphic to $\MOD{g}{P} / \FS_P$, $\MUD{g}{P} / \FS_P$ respectively.
\end{coro}

\begin{rem}
By Corollaries~\ref{maincoro1} and \ref{maincoro2} the surjective maps $\ol{\pi}: \MOC{g}{P} \to \MO{g}{P}$, $\ul{\pi}: \MUC{g}{P} \to \MU{g}{P}$ are homotopy equivalences and thus a chain complex computing the homology of the domains will compute the homology of the target spaces.
\end{rem}

The spaces $\MUD{g}{P} / \FS_P$ are the decorated analogues of the spaces used in \cite{cos:gp} to construct a solution to the quantum master equation. Using the last corollary and extending the previous remark it might be possible to describe a solution to the master equation in terms of ribbon graphs.

Another interesting question is how to extend the present result for the moduli of bordered Riemann surfaces and whether that also yields a combinatorial solution to the quantum master equation as in \cite{hvz}.

\bibliographystyle{amsalpha}
\bibliography{cms}

\end{document}